\documentclass[11pt]{article}
\usepackage{stmaryrd}
\usepackage{amsmath}
\usepackage{amsfonts}
\usepackage{amsthm}
\usepackage{txfonts}
\usepackage{hyperref}
\usepackage{graphicx,amssymb,amsfonts,amsmath,amscd}
\usepackage{	textcomp}
\usepackage[all,ps,cmtip]{xy}
\usepackage{amscd}
\usepackage{xspace}

\usepackage{mathrsfs}

\def\bC{{\bf C}}

\def\bN{{\bf N}}

\def\bP{{\bf P}}
\def\bQ{{\bf Q}}

\def\bZ{{\bf Z}}

\def\C{\bC}

\def\Z{\bZ}
\def\N{\bN}
\def\Q{\bQ}

\def\sO{{\mathscr O}}

\def\sS{{\mathscr S}}

\def\sZ{{\mathscr Z}}

\def\ie{\textit {i.e.}~}
\def\cf{\textit {cf.}~}

\def\apriori{\textit{a priori} }
\def\iff{if and only if }

\def\CH{\mathop{\rm CH}\nolimits} % Chow groups
\def\ch{\mathop{\rm ch}\nolimits} % Chern character
\def\cl{\mathop{\rm cl}\nolimits} %cycle class map
\def\dual{\mathop{^\vee}\nolimits} % dual
\def\Gr{\mathop{\rm Gr}\nolimits} % Grassmannian
\def\Hom{\mathop{\rm Hom}\nolimits}
\def\id {\mathop{\rm id}\nolimits} %identity
\def\im{\mathop{\rm Im}\nolimits} % image, imaginary part is \Im
\def\Jac{\mathop{\rm Jac}\nolimits} % Jacobian of curves
\def\mod{\mathop{\rm mod} \nolimits} % modulo
\def\P{\mathop{\bP}\nolimits} % projective space, projectivization
\def\Pic{\mathop{\rm Pic}\nolimits} % Picard group
\def\pr{\mathop{\rm pr}\nolimits} % projection
\def\pt{\mathop{\rm pt}\nolimits} % point
\def\rank{\mathop{\rm rank}\nolimits}
\def\Sym{\mathop{\rm Sym}\nolimits} % Symmetric products
\def\bar{\overline}
\def\inj{\hookrightarrow}
\def\surj{\twoheadrightarrow}

\def\lra{\xrightarrow}

\def\cart{\ar@{}[dr]|\square} % cartesian diagrams, write it after the left-up term to produce a square in the middle of the diagram
\def\dual{^\vee} % dual
\def\isom{\simeq} %isomorphism
\def\tilde{\widetilde}

\newtheorem{thm}{Theorem}[section]
\newtheorem{prop}[thm]{Proposition}
\newtheorem{lemma}[thm]{Lemma}
\newtheorem{cor}[thm]{Corollary}

\newtheorem{defi}[thm]{Definition}
\newtheorem{rmk}[thm]{Remark}

\newtheorem{def-prop}[thm]{Definition-Proposition}
\newtheorem{prop-def}[thm]{Proposition-Definition}

\begin{document}

\title{Decomposition of small diagonals and Chow rings of hypersurfaces and Calabi-Yau complete intersections}
\author{Lie Fu}
\date{ }

\maketitle

\begin{abstract}
On one hand, for a general Calabi-Yau complete intersection $X$, we establish a decomposition, up to rational equivalence, of the small diagonal in $X\times X\times X$, from which we deduce that any \emph{decomposable} 0-cycle of degree 0 is in fact rationally equivalent to 0, up to torsion. On the other hand, we find a similar decomposition of the smallest diagonal in a higher power of a hypersurface, which provides us an analogous result on the multiplicative structure of its Chow ring.
\end{abstract}

\setcounter{section}{-1}
\section{Introduction}
For a given smooth projective complex algebraic variety $X$, we can
construct very few subvarieties or algebraic cycles of $X\times X$
in an \apriori fashion. Besides the divisors and the exterior
products of two algebraic cycles of each factor, the diagonal
$\Delta_X:=\left\{(x,x)~|~x\in X\right\}\subset X\times X$ is
essentially the only one that we can canonically construct in
general. Despite of its simplicity, the diagonal in fact contains a
lot of geometric information of the original variety. For instance,
its normal bundle is the tangent bundle of $X$; its
self-intersection number is its topological Euler characteristic and
so on. Besides these obvious facts, we would like to remark that the
Bloch-Beilinson-Murre conjecture (\cf \cite{MR923131}
\cite{MR1225267} \cite{MR1265533}), which is considered as one of
the deepest conjectures in the study of algebraic cycles, claims a
conjectural decomposition of the diagonal, up to rational
equivalence, as the sum of certain orthogonal idempotent
correspondences related to the Hodge structures on its Betti
cohomology groups.

The idea of using decomposition of diagonal to study algebraic cycles is initiated by Bloch and Srinivas \cite{MR714776}; we state their main theorem in the following form.
\begin{thm}[Bloch, Srinivas \cite{MR714776}]\label{BS}
Let $X$ be a smooth projective complex algebraic variety of
dimension $n$. Suppose that $\CH_0(X)$ is supported on a closed
algebraic subset $Y$, \ie the natural morphism $\CH_0(Y)\to
\CH_0(X)$ is surjective. Then there exist a positive integer $m\in
\N^*$ and a proper closed algebraic subset $D\subsetneqq X$, such
that in $\CH_n(X\times X)$ we have
\begin{equation} \label{BSdecomp}
 m\cdot \Delta_X=\sZ_1+\sZ_2
\end{equation}
where $\sZ_1$ is supported on $Y\times X$, and $\sZ_2$ is supported on $X\times D$.
\end{thm}

The above decomposition in the case that $Y$ is a point, or equivalently $\CH_0(X)\isom \Z$ by degree map, is further generalized by Paranjape \cite{MR1283872} and Laterveer \cite{MR1669995} for varieties of small Chow groups in the following form.
\begin{thm}[Paranjape \cite{MR1283872}, Laterveer \cite{MR1669995}]
Let $X$ be a smooth projective $n$-dimensional variety. If the cycle
class map $\cl: \CH_i(X)_\Q\to H^{2n-2i}(X, \Q)$ is injective for
any $0\leq i\leq c-1$. Then there exist a positive integer $m\in
\N^*$, a closed algebraic subset $T$ of dimension $\leq n-c$, and
for each $i\in\{0,1, \cdots,c-1\}$, a pair of closed algebraic
subsets $V_i, W_i$ with $\dim V_i=i$ and $\dim W_i=n-i$, such that
in $\CH_n(X\times X)$, we have
\begin{equation} \label{PLdecomp}
 m\cdot \Delta_X=\sZ_0+\sZ_1+\cdots+\sZ_{c-1}+\sZ'
\end{equation}
where $\sZ_i$ is supported on $V_i\times W_i$ for any $0\leq i<c$, and $\sZ'$ is supported on $X\times T$.
\end{thm}

For applications of such decompositions, the point is that we
consider (\ref{BSdecomp}) and (\ref{PLdecomp}) as equalities of
correspondences from $X$ to itself, which yield decompositions of
the identity correspondence. This point of view allows us to deduce
from (\ref{BSdecomp}) and (\ref{PLdecomp}) many interesting results like generalizations of Mumford's theorem (\cf \cite{MR0249428} \cite{MR714776} \cite{MR1997577}).
% For example:
% \begin{itemize}
%  \item The Hodge conjecture for degree $4$ cohomology for smooth projective varieties with $\CH_0$ supported on some $3$-dimensional algebraic subset.
%  \item \emph{Generalized Mumford theorem I:} if $X$ is a smooth projective variety such that $\CH_0(X)$ is supported on some $k$-dimensional algebraic subset, then for any $j>k$, $H^j(X,\Q)$ is supported on a divisor, \ie the restriction map $H^j(X,\Q)\to H^j(X\backslash D,\Q)$ is zero for a closed algebraic subset $D$ of codimension 1. In particular, for any $j>k$, the Hodge coniveau of $H^j(X,\Q)$ is at least 1, that is $H^0(X,\Omega_X^j)=0$.
%  \item \emph{Generalized Mumford theorem II:} if  $X$ is a smooth projective variety such that the cycle class map $\cl: \CH_i(X)_\Q \to H^{2n-2i}(X,\Q)$ is injective for $0\leq i<c$. Then for any $j$, the transcendental part of the cohomology group $H^j(X,\Q)^{\perp_\text{alg}}$ is supported on some closed algebraic subset of codimension at least $c$. In particular, the Hodge coniveau of $H^j(X,\Q)^{\perp_\text{alg}}$ is at least $c$ for any $j$.
% \end{itemize}
% The converse of the generalized Mumford theorems is the so-called
% generalized Bloch conjecture (\cf \cite{MR1997577}
% \cite{HodgeBloch}).\\

Most of this paper is devoted to the study of the class of the \emph{small diagonal}
\begin{equation}
 \delta_X:=\{(x,x,x)\in X^3~~|~~x\in X\}
\end{equation}
in $\CH_n(X^3)_\Q$, where $X$ is an $n$-dimensional Calabi-Yau variety. The interest of the study is motivated by the obvious fact that while the diagonal seen as a self-correspondence of $X$ controls $\CH^*(X)_\Q$ as an additive object, the small diagonal seen as a correspondence between $X\times X$ and $X$ controls the multiplicative structure of $\CH^*(X)_\Q$.

The first result in this direction is due to Beauville and Voisin \cite{MR2047674}, who find a
decomposition of the \emph{small} diagonal $\delta_S:=\{(x,x,x)~|~
x\in S\}$ in $\CH_2(S\times S\times S)$ for $S$ an algebraic K3 surface.
\begin{thm}[Beauville, Voisin \cite{MR2047674}] \label{BV}
Let $S$ be a projective K3 surface, and $c_S\in \CH_0(S)$ be the
well-defined\footnote{The fact that $c_S$ is well-defined relies on
the result of Bogomolov-Mumford (\cf the appendix of
\cite{0557.14015}) about the existence of rational curves in any
ample linear system, \cf \cite{MR2047674}.} 0-dimensional cycle of
degree 1 represented by any point lying on any rational curve of
$S$. Then in $\CH_2(S\times S\times S)$, we have
\begin{equation}\label{BVdecomp}
\delta_S=\Delta_{12}+\Delta_{13}+\Delta_{23}-S\times c_S\times c_S-c_S\times S\times c_S-c_S\times c_S\times S
\end{equation}
where $\Delta_{12}$ is represented by $\left\{(x,x,c_S)~|~x\in S\right\}$, and $\Delta_{13}$, $\Delta_{23}$ are defined similarly.
\end{thm}

As is mentioned above, we regard (\ref{BVdecomp}) as an
equality of correspondences from $S\times S$ to $S$. Applying this
to a product of two divisors $D_1\times D_2$, one can recover the
following corollary, which is in fact a fundamental observation in
\cite{MR2047674} for the proof of the above theorem.

\begin{cor}[Beauville, Voisin \cite{MR2047674}]\label{BVCor}
Let $S$ be a projective K3 surface. Then the intersection product of any two divisors is always proportional to the class $c_S$ in $\CH_0(S)$, \ie
$$\im\left(\Pic(S)\times \Pic(S)\lra{\bullet}\CH_0(S)\right)=\Z\cdot c_S.$$
\end{cor}

As is pointed out in \cite{MR2047674}, this corollary is somehow
surprising because $\CH_0(S)$ is \apriori very huge (`of infinite
dimension' in the sense of Mumford, \cf \cite{MR0249428}).\\

The next result in this direction, which is also the starting point of this paper, is the following partial generalization of Theorem \ref{BV} due to Voisin:

% In this K3 surface case, we want to highlight the following
%
% \vspace{0.1cm}
% \noindent\textbf{Principle}: The small diagonal controls the intersection product of the Chow ring.
% \vspace{0.1cm}
%
% The objective of this paper is to pursue this principle. Our
% starting point is a recent result of Voisin in this direction, which
% is a partial generalization of Theorem \ref{BV}.

\begin{thm}[Voisin \cite{VoiFamilyK3}] \label{V}
Let $X\subset\P^{n+1}$ be a general smooth hypersurface of Calabi-Yau type, that is, the degree of $X$ is $n+2$. Let $h:=c_1\big(\sO_X(1)\big)\in \CH^1(X)$ be the hyperplane section class, $h_i:=\pr_i^*(h)\in \CH^1(X^3)$ for $i=1,2,3$, and $c_X:=\frac{h^n}{n+2}\in \CH_0(X)_\Q$ be a $\Q$-0-cycle of degree 1.  Then we have a decomposition of the small diagonal in $\CH_n(X^3)_\Q$
\begin{equation}\label{Vdecomp}
 \delta_X=\frac{1}{(n+2)!}\Gamma+\Delta_{12}+\Delta_{13}+\Delta_{23}+P(h_1, h_2, h_3),
\end{equation}
where $\Delta_{12}=\Delta_X \times c_X$, and $\Delta_{13}$, $\Delta_{23}$ are defined similarly; $P$ is a homogeneous polynomial of degree $2n$; and  $\Gamma:=\bigcup_{t\in F(X)}\P^1_t\times\P^1_t\times\P^1_t\subset X^3$, where $F(X)$ is the variety of lines of $X$, and $\P^1_t$ is the line corresponding to $t\in F(X)$.
\end{thm}
Applying (\ref{Vdecomp}) as an equality of correspondences, she
deduces the following
\begin{cor}[Voisin \cite{VoiFamilyK3}] \label{VCor}
In the same situation as the above theorem, the intersection product
of any two cycles of complementary and strictly positive
codimensions is always proportional to $c_X$ in $\CH_0(X)_\Q$, \ie
for any $i,j\in \N^*$ with $i+j=n$, we have
$$\im\left(\CH^i(X)_\Q\times \CH^j(X)_\Q\lra{\bullet}\CH_0(X)_\Q\right)= \Q\cdot
c_X.$$ In particular, for any $ i=1,\cdots,m$, let $Z_i, Z'_i$ be
algebraic cycles of strictly positive codimension with $\dim
Z_i+\dim Z'_i=n$, then any equality on the cohomology level
$\sum_{i=1}^m\lambda_i[Z_i]\cup[Z'_i]=0$ in $H_0(X,\Q)$ is in fact
an equality modulo rational equivalence: $\sum_{i=1}^m\lambda_i
Z_i\bullet Z'_i=0$ in $\CH_0(X)_{\Q}$.
\end{cor}

The main results of this paper are further generalizations of
Voisin's theorem and its corollary in two different
directions.\\

The first direction of generalization is about smooth Calabi-Yau
complete intersections in projective spaces:

\begin{thm}[=Theorem \ref{main1}+Theorem \ref{main2}+Theorem \ref{main1CI}+Theorem \ref{main2CI}]\label{thmA}
  Let $E$ be a rank $r$ vector bundle on $\P^{n+r}$ satisfying the following positivity condition:\\
  \noindent $(*)$ The evaluation map $H^0(\P^{n+r},E)\to \oplus_{i=1}^3E_{y_i}$ is surjective for any three collinear points $y_1, y_2, y_3\in
  \P^{n+r}.$\\
  Let $X$ be the zero locus of a general section of $E$. Suppose that the canonical bundle of $X$ is
  trivial, \ie $\det(E)\isom \sO_{\P^{n+r}}(n+r+1)$. Then we have:
\begin{enumerate}
  \item[(i)] There are symmetric homogeneous polynomials with $\Z$-coefficients\footnote{See
Theorem \ref{main1} for more about their coefficients.} $Q, P$, such
that in $\CH_n(X^3)$:
  \begin{equation*}
a_0\deg(X)\cdot\delta_X=
\Gamma+j_{12*}(\sZ)+j_{13*}(\sZ)+j_{23*}(\sZ)+P(h_1,h_2,h_3)
  \end{equation*}
  \begin{equation*}
    \text{with}~~~ \sZ:=Q(h_1, h_2)~~~\text{in}~ \CH_n(X\times X),
  \end{equation*}
  where $\Gamma$ is defined as in Theorem \ref{V}; $h_i\in \CH^1(X^3)$ or
$\CH^1(X^2)$ is the pull-back of $h=c_1(\sO_X(1))$ by the $i^{th}$
projection; and the inclusions of big diagonals are given by:
\begin{eqnarray*}
  & X^2 &\inj X^3\\
  j_{12}: &(x,x')&\mapsto  (x,x,x')\\
  j_{13}: &(x,x')&\mapsto  (x,x',x)\\
  j_{23}: &(x,x')&\mapsto  (x',x,x).\\
\end{eqnarray*}
  \item[(ii)] If the coefficients of $Q$ determined by (\ref{Q}) in Proposition \ref{predecompProp} satisfy $a_0\neq 0$ and $a_1\neq a_0$, then
the intersection product of two cycles of strictly positive
complementary codimension is as simple as possible: for any $k,l\in
\N^*$ with $k+l=n$,
$$\im\left(\bullet: \CH^k(X)_{\Q}\times \CH^l(X)_{\Q}\to \CH_0(X)_{\Q}\right)=\Q\cdot h^n,$$
where $h=c_1(\sO_X(1))\in \CH^1(X)$.
\item[(iii)] The conditions $a_0\neq 0$ and $a_1\neq a_0$ are always satisfied in the splitting case, that is, $E=\oplus_{i=1}^r\sO_{\P}(d_i)$  with $d_i\geq 2$,
hence the conclusions of (i) and (ii) stated above hold for for any general Calabi-Yau complete intersections\footnote{See Theorem \ref{main1CI} and Theorem \ref{main2CI} for the precise statements.}.
\end{enumerate}
\end{thm}

A word for the conditions $a_0\neq 0$ and $a_1\neq a_0$ appeared above: they are (mild) numerical conditions on the Chern
numbers of the vector bundle $E$, which are satisfied for complete
intersections (see (iii)). While for the full generality as in (i) and
(ii), it seems that we have to add them, but presumably they are
automatic under the positivity condition $(*)$ (see Remark \ref{nonsplit}).

Again, the above degeneration property of the intersection product is remarkable since $\CH_0$ is very huge. To emphasize the principle that the small diagonal controls the intersection product, we point out that
part (ii) of the preceding theorem is obtained by applying the
equality of part (i) as correspondences to an
exterior product of two algebraic cycles.

We will make a comparison of our result (iii) for Calabi-Yau complete intersections with Beauville's `weak splitting principle' in \cite{MR2187148} for holomorphic symplectic varieties, which says\footnote{This is the strengthened version conjectured by Voisin in \cite{MR2435839}.} that
any polynomial relation between the \emph{cohomological} Chern
classes of lines bundles and the tangent bundle holds already for
their \emph{Chow-theoretical} Chern classes. In our case we prove
more: any \emph{decomposable} $\Q$-coefficient 0-cycle is rational
equivalent to 0 if and only if it has degree 0, see Remark \ref{BB}.

We also give in \S\ref{examplesection} an example of a surface $S$ in $\P^3$ of general type,
such that $$\im\left(\bullet: \CH^1(S)_{\Q}\times \CH^1(S)_{\Q}\to
\CH_0(S)_{\Q}\right)\supsetneqq \Q\cdot h^2,$$ where
$h=c_1(\sO_S(1))\in \CH^1(S)$, in the contrary of the result for K3
surfaces proved in \cite{MR2047674}. This example supports the
feeling that the Calabi-Yau condition gives
some strong restrictions on the multiplicative structure of the Chow ring.\\

The second direction of generalization is about higher powers
($\geq 3$) of hypersurfaces with ample or trivial
canonical bundle. The objective is to decompose the smallest
diagonal in a higher self-product, and deduce from it some
implication on the multiplicative structure of the Chow ring of the
variety. That such a decomposition should exist was suggested by Nori to be the natural generalization of Theorem \ref{V}. We also refer to \cite{MR2339831} for similar results in the case of curves. Now we state our result precisely:
\begin{thm}[=Theorem \ref{main3} + Theorem \ref{main3cor}]\label{thmB}
  Let $X$ be a smooth hypersurface in $\P^{n+1}$ of degree $d$ with
  $d\geq n+2$. Let $k=d+1-n\geq 3$. Then

  \begin{enumerate}
    \item[(i)] One of the following two cases occurs:
    \begin{enumerate}
    \item There exist rational numbers $\lambda_j$ for $j=2,\cdots, k-1$, and a symmetric
  homogeneous polynomial $P$ of degree $n(k-1)$, such that in
  $\CH_n(X^k)_{\Q}$ we have:
\begin{equation*}
\delta_X=(-1)^{k-1}\frac{1}{d!}\cdot\Gamma+\sum_{i=1}^kD_i+\sum_{j=2}^{k-2}\lambda_j\sum_{|I|=j}D_I+
P(h_1,\cdots,h_k),
\end{equation*}
where the bigger diagonals $D_I$ are defined in (\ref{DefD1}) or (\ref{DefD2}), and $D_i:=D_{\{i\}}$.\\
Or

\item There exist a (smallest) integer $3\leq l<k$, rational numbers $\lambda_j$ for $j=2,\cdots, l-2$,
and a symmetric homogeneous polynomial $P$ of degree $n(l-1)$, such
that in
  $\CH_n(X^l)_{\Q}$ we have:
\begin{equation*}
\delta_X=\sum_{i=1}^lD_i+\sum_{j=2}^{l-2}\lambda_j\sum_{|I|=j}D_I+
P(h_1,\cdots,h_l).
\end{equation*}
  \end{enumerate}
Moreover, $\Gamma=0$ if $d\geq 2n$.

    \item[(ii)] For any strictly positive
  integers $i_1, i_2,\cdots, i_{k-1}\in \N^*$ with $\sum_{j=1}^{k-1}i_j=n$, the image
  $$\im\left(\CH^{i_1}(X)_\Q\times\CH^{i_2}(X)_\Q\times\cdots\times\CH^{i_{k-1}}(X)_\Q\lra{\bullet} \CH_0(X)_\Q\right)=\Q\cdot h^{n}$$
  \end{enumerate}
\end{thm}
We remark that when $d>n+2$,  part (ii) of the previous theorem is
in fact implied by the Bloch-Beilinson conjecture, see Remark \ref{rmkcor}.\\

The main line of the proofs of the above theorems is the same as in
Voisin's paper \cite{VoiFamilyK3}: one proceeds in three steps:
\begin{itemize}
  \item Firstly, one `decomposes' the
class of $\Gamma$ (see Theorem \ref{V} for its definition)
restricting to the complementary of the small diagonal, by means of
a careful study of the geometry of collinear points on the variety.
  \item Secondly, by the localization exact
sequence for Chow groups, we obtain a decomposition of certain
multiple of the small diagonal in terms of $\Gamma$ and other cycles
of diagonal type or coming from the ambiant space.
  \item Thirdly, once we have a decomposition of small diagonal,
  regarded as an equality of correspondences, we will draw
  consequences on the multiplicative structure of the Chow ring.
\end{itemize}
We remark that at some point of the proof, one should verify that the `multiple' appeared in the decomposition is non-zero to get a genuine decomposition of the small diagonal, and also some coefficients should be distinct to deduce the desired conclusion on the multiplicative structure of Chow rings. These are too easy to be noticed in \cite{VoiFamilyK3}, but become the major difficulties in
  our present paper.

The paper consists of two parts. The first part deals with the first
direction of generalizations explained above, where we start by the
geometry of collinear points to deduce a decomposition of a certain
multiple of the small diagonal; then we deduce from it our result on the multiplicative structure of Chow rings; after that  we treat the complete intersection case to obtain the main results; finally we construct an example of surface in $\P^3$ such that the image of intersection product of line bundles is `non-trivial', on the contrary of the Calabi-Yau case. The second part deals with the second direction of generalizations explained above, where we also follow the line of
geometry of collinear points, decomposition of the smallest
diagonal, and finally consequence on Chow ring's structure.

We will work over the complex numbers throughout this paper for
simplicity, but all the results and proofs go through for any
uncountable algebraic closed field of characteristic zero.

\paragraph*{Acknowledgements:}
I would like to thank my thesis advisor Claire Voisin for bringing me into attention of her paper
\cite{VoiFamilyK3}, and for her help through the construction of the example in \S\ref{examplesection} as well as her kindness to share with me a
remark of Nori on \cite{VoiFamilyK3}, which motivates \S2.

\section{Calabi-Yau complete intersections}

Let $E$ be a rank $r$ vector bundle on the complex projective space
$\P:=\P^{n+r}$. We always make the following

\vspace{0.2cm} \noindent\textbf{Positivity Assumption ($*$):} The
evaluation map $H^0(\P,E)\to \oplus_{i=1}^3E_{y_i}$ is surjective
for any three collinear points $y_1, y_2, y_3$ in $\P$, where $E_y$
means the fibre of $E$ over $y$. \vspace{0.2cm}

\noindent We note that this condition implies in particular:\\
($*'$) $E$ is globally generated.\\
($*''$) The restriction of $E(-2):=E\otimes \sO_{\P}(-2)$ to each
line is globally generated\footnote{Equivalently speaking, along any
line $\P^1$, the splitting type
$E|_{\P^1}=\oplus_{i=1}^r\sO_{\P^1}(d_i)$ satisfies $d_1\geq d_2\geq
\cdots\geq d_r\geq 2$.}.\\

Let $d\in \N^*$ be such that $\det(E)=\sO_{\P}(d)$. Let $f\in H^0(\P^{n+r}, E)$ be a general global section of $E$, and $$X:=V(f)\subset \P^{n+r}$$ be the subscheme of $\P$ defined by $f$.

\begin{rmk}\upshape
The assumption that $f$ is generic implies that $X$
 is smooth of expected dimension $n$.
% In fact, taking the section $f$ \emph{general} will provide examples
% satisfying the above requirement that $X$ is smooth of expected
% dimension $n$.
Indeed, consider the incidence variety
$$I:=\left\{([f],x)\in \P\left(H^0(\P^{n+r},E)\right)\times
\P^{n+r}~|~ f(x)=0\right\}.$$ Let $q:I\to
\P\left(H^0(\P^{n+r},E)\right)$ and $p: I\to \P^{n+r}$ be the two
natural projections. The global generated property of $E$ implies
that $p$ is a projective bundle, thus $I$ is smooth of dimension
$h^0(\P^{n+r},E)-1+n$. Therefore the theorem of generic smoothness
applied to $q$ proves
 the assertion.
\end{rmk}

We are interested in the case when $X$ is of \emph{Calabi-Yau type}: $K_X=0$, or equivalently,

\vspace{0.2cm}
\noindent\textbf{Calabi-Yau Assumption:} $d=n+r+1$.
\vspace{0.2cm}

Throughout this section, we will always work in the above setting. A
typical example of such situation is when
$E=\bigoplus_{i=1}^r\sO_{\P}(d_i)$ with $d_1\geq d_2\geq\cdots\geq
d_r\geq2$, $X$ is hence a smooth Calabi-Yau complete intersection of
multi-degree $(d_1,\cdots,d_r)$ and $d=\sum_{i=1}^rd_i=n+r+1$. Since
the case of K3 surfaces is well treated in \cite{MR2047674}, we
assume $n\geq 3$ from now on.

\subsection{Decomposition of small diagonals}
Like in the paper \cite{VoiFamilyK3}, our strategy is to express the class of the small diagonal by investigating the lines
contained in $X$. The decomposition result is Corollary \ref{decomp} and see Theorem \ref{main1} for its more precise form.

Let $G:=\Gr(\P^1, \P^{n+r})$ be
the Grassmannian of projective lines in $\P$, and for a point $t\in
G$, we denote by $\P^1_t$ the corresponding line. Define the
\emph{variety of lines} of $X$:
$$F(X):=\left\{t\in G~|~ \P^1_t\subset X\right\}.$$
\begin{lemma} \label{F(X)}
(If $f$ is general,) $F(X)$ is non-empty, smooth and of dimension
$n-3$.
\end{lemma}
\begin{proof}
Let $p:L\to G$ be the universal line over the Grassmannian, and
$q:L\to \P$ be the natural morphism. Then the section $f$ of $E$
gives rise to a section $\sigma_f$ of the vector bundle $p_*q^*E$ on
$G$, and $F(X)=(\sigma_f=0)$ by definition. Thanks to our positivity
assumption $(*'')$, the vector bundle $E|_{\P^1_t}$ has splitting
type $\oplus_{i=1}^r\sO_{\P^1}(d_i)$ satisfying $d_1\geq d_2\geq
\cdots\geq d_r\geq 2$ for any $t\in G$, thus
$$\dim H^0(\P^1_t, E|_{\P^1_t})=\sum_{i=1}^r(d_i+1)=d+r=n+2r+1.$$
The last equality comes from the Calabi-Yau assumption. Therefore by Grauert's base-change theorem, $p_*q^*E$ is a vector bundle of rank $n+2r+1$, and thus the \emph{expected} dimension of $F(X)$ is
 $$\exp\dim F(X)=\dim G-\rank(p_*q^*E)=2(n+r-1)-(n+2r+1)=n-3\geq 0.$$
Now we can use the results in for example \cite{0913.14015} to conclude.
%\footnote{The relevant invariants defined there read in our case %$\delta=\delta_{-}=n-3\geq 0$.}.
\end{proof}
We define
\begin{equation}\label{Gamma}
 \Gamma:=\bigcup_{t\in F(X)}\P^1_t\times\P^1_t\times\P^1_t\subset X^3.
\end{equation}
It is then an $n$-dimensional subvariety of $X^3$.

% When $f$ is not assumed to be general,  as long as $X$ is smooth, we
% can always define $F(X)$ and $\Gamma$ as algebraic cycles modulo
% rational equivalence, which coincide with their fundamental classes
% if $f$ is general. The reason is that $F(X)$ and $\Gamma$, modulo
% rational equivalence, could have been defined in a purely
% intersection theoretical way, that is:
% $$F(X):=c_{n+2r+1}(p_*q^*E)\in \CH_{n-3}(G);$$
% where the morphisms are defined in the following diagram, and $S$ is
% the universal rank 2 vector bundle on $G$.
% \begin{displaymath}
% \xymatrix{
%  \P(S) \ar[r]_{q} \ar[d]_{p} & \P\\
%   G
%   }
% \end{displaymath}
%
% Similarly, we define $\Gamma:= i^!q^3_*p^{3*}(F(X))\in \CH_n(X^3)$,
% where the morphisms are defined in the following diagram:
% \begin{displaymath}
% \xymatrix{
%   & X^3\ar[d]^{i}\\
%  \P(S)^{\times_G 3} \ar[r]_{q^3} \ar[d]_{p^3} & \P^{\times 3}\\
%   G
%   }
% \end{displaymath}
%
%
%
% From now on, $\Gamma$ will be understood as the $n$-dimensional
% cycle of $X^3$ defined in this way.

To get a decomposition of the
small diagonal, we first decompose or calculate the class of
$\Gamma_o$ in $\CH_n(X^3\backslash\delta_X)$, where $\Gamma_o$ is
the restriction of $\Gamma$ to $X^3\backslash\delta_X$.

Before doing so, let us make some preparatory geometric constructions. Let $\delta_{\P}$ be the small diagonal of $\P^{\times 3}:=\P\times \P\times \P$. Define the following closed subvariety of $\P^{\times 3}\backslash\delta_{\P}$:
$$W:= \left\{(y_1,y_2,y_3)\in \P^{\times 3}~|~y_1, y_2, y_3~ \text{are collinear}\right\}\backslash\delta_{\P}.$$
In other words, if we denote by $L\to G$ the universal line over the
Grassmannian $G$ of projective lines, then in fact $W= L\times_G
L\times_G L\backslash \delta_{L}$. In particular $W$ is a smooth
variety with $$\dim W=\dim G+3=2n+2r+1.$$

Similarly, let $\delta_X$ be the small diagonal of $X^3$. We define a closed subvariety $V$ of $X^3\backslash\delta_X$ by taking the closure of
$V_o:=\left\{(x_1,x_2,x_3)\in X^3~|~x_1, x_2, x_3~ \text{are collinear and distinct}\right\}$
in $X^3\backslash\delta_X$:
$$V:=\bar{V_o}.$$
We remark that the boundary $\partial V:=V\backslash V_o$ consists, up to a permutation of the three coordinates, of points of the form $(x,x,x')$ with $x\neq x'$ such that the line joining $x, x'$ is tangent to $X$ at $x$. \\
We will also need the `big' diagonals in $X^3\backslash\delta_X$:
$$\Delta_{12}:=\left\{(x,x,x')\in X^3~|~ x\neq x'\right\},$$
and $\Delta_{13}, \Delta_{23}$ are defined in the same way.

\begin{lemma} \label{intersectionlemma}
Consider the intersection of $W$ and $X^3\backslash\delta_X$ in $\P^{\times 3}\backslash\delta_{\P}$. The intersection scheme has four irreducible components:
\begin{equation}\label{intersection}
 W\cap (X^3\backslash\delta_X)= V\cup \Delta_{12}\cup\Delta_{13}\cup\Delta_{23}.
\end{equation}
The intersection along $V$ is transversal, in particular $\dim V=2n-r+1$. The intersection along $\Delta_{ij}$ is not proper, having excess dimension $r-1$, but the multiplicity of $\Delta_{ij}$ in the intersection scheme is 1, where $1\leq i<j\leq 3$. In particular, the intersection scheme is reduced and the above identity (\ref{intersection}) also holds scheme-theoretically:
\begin{displaymath}
\xymatrix{
V\cup \Delta_{12}\cup\Delta_{13}\cup\Delta_{23} \cart \ar[r] \ar[d] & X^3\backslash\delta_X \ar[d]\\
W \ar[r] &  \P^{\times 3}\backslash\delta_{\P}
}
\end{displaymath}
\end{lemma}
\begin{proof}
 It is obvious that (\ref{intersection}) holds set-theoretically. To verify (\ref{intersection}) scheme-theoretically, let $$W_o:=\left\{(y_1,y_2,y_3)\in \P^{\times 3}~|~ y_1, y_2, y_3 ~\text{are collinear and distinct}\right\}.$$ Consider the incidence variety $$I:=\left\{\big([f],(y_1,y_2,y_3)\big)\in \P\left(H^0(\P^{n+r},E)\right)\times W_o~|~ f(y_1)=f(y_2)=f(y_3)=0\right\}.$$ Let $q:I\to \P\left(H^0(\P^{n+r},E)\right)$ and $p:I\to W_o$ be the two natural projections. The positivity assumption $(*)$ means precisely that $p$ is a $\P^{h^0(\P,E)-1-3r}$-bundle, therefore $I$ is smooth of dimension $h^0(\P,E)+2n-r$. Since for general $f$, the corresponding variety $X$ contains a line (Lemma \ref{F(X)}), $q$ is of dominant. By the theorem of generic smoothness, the fibre of $q$, which is exactly $V_o$, is smooth of dimension $2n-r+1$. In particular, $V_o$ is reduced, of locally complete intersection in $W_o$ of codimension $3r$ as expected. In other words, the intersection is transversal along a general point\footnote{This suffices for the scheme-theoretical assertions concerning $V$, since we work over the complex numbers, there are enough (closed) points such that any algebraic condition satisfied by a general point is also satisfied by the generic point.}  of $V$.

The assertions concerning the big diagonals are easier: by passing to a general point of $\Delta_{ij}$, it amounts to prove that the intersection scheme of $\Delta_{\P}$ and $X^2$ in $\P^{\times 2}$ is $\Delta_X$ with multiplicity 1, (and of excess dimension $r$).
\end{proof}

Now we construct a vector bundle $F$ on $W$. Let $S$ be the
tautological rank 2 vector bundle on $W$, with fibre $S_{y_1y_2y_3}$
over a point $(y_1,y_2,y_3)\in W$ the 2-dimensional vector space
corresponding to the projective line $\P^1_{y_1y_2y_3}$ determined
by these three collinear points. Therefore $p:\P(S)\to W$ is the
$\P^1$-bundle of universal line, and it admits three tautological
sections $\sigma_i: W\to \P(S)$ determined by the points $y_i$,
where $i=1,2,3$. Let $q: \P(S)\to \P^{n+r}$ be the natural morphism.
We summarize the situation by the following diagram:
\begin{displaymath}
 \xymatrix{
\P(S) \ar[r]^{q} \ar[d]^{p} & \P\\
W\ar@/^/[u]<3ex>^{\sigma_i} \ar@/^/[u]<2ex> \ar@/^/[u]<1ex>
}
\end{displaymath}
Let $D_i$ be the image of section $\sigma_i$, which is a divisor of $\P(S)$, for $i=1,2,3$. We define the following sheaf on $W$:
\begin{equation}\label{F}
F:=p_*(q^*E\otimes \sO_{\P(S)}(-D_1-D_2-D_3))
\end{equation}

\begin{lemma}\label{FLemma}
 $F$ is a vector bundle on $W$ of rank $n-r+1$, with fibre $$F_{y_1y_2y_3}=H^0(\P^1_{y_1y_2y_3}, E|_{\P^1_{y_1y_2y_3}}\otimes\sO(-y_1-y_2-y_3)).$$
\end{lemma}
\begin{proof}
 For any $(y_1, y_2, y_3)\in W$, it is obvious that the restriction of the vector bundle $q^*E\otimes \sO(-D_1-D_2-D_3)$ to the fibre $p^{-1}(y_1, y_2, y_3)=:\P^1_{y_1y_2y_3}$ is exactly $E|_{\P^1_{y_1y_2y_3}}\otimes\sO(-y_1-y_2-y_3)$. By the positivity assumption $(*)$, the splitting type of $E$ at $\P^1_{y_1y_2y_3}$ is $\oplus_{i=1}^r\sO(a_i)$ with $a_1\geq a_2\geq \cdots\geq a_r\geq 2$, we find that $$h^0\left(\P^1_{y_1y_2y_3}, E|_{\P^1_{y_1y_2y_3}}\otimes\sO(-y_1-y_2-y_3)\right)=h^0\left(\P^1, \oplus_{i=1}^r\sO(a_i-3)\right)=\sum_{i=1}^r(a_i-2)=d-2r=n-r+1,$$
which is independent of the point of $W$. Now the lemma is a consequence of Grauert's base-change theorem.
\end{proof}

Here is the motivation to introduce the vector bundle $F$: the section $f\in H^0(\P, E)$ gives rise to a section $s_f\in
H^0(V,F|_V)$ in the way that for any $(x_1,x_2,x_3)\in V$ the value
$s_f(x_1,x_2,x_3)$ is simply given by $f|_{\P^1_{x_1x_2x_3}}\in
F_{x_1x_2x_3}=H^0(\P^1_{x_1x_2x_3}, E|_{\P^1_{x_1x_2x_3}}\otimes
\sO(-x_1-x_2-x_3))$ (see Lemma \ref{FLemma}), because $f$ vanishes
on $x_i$ by definition. As a result,
\begin{lemma}\label{gamma2}
$c_{n-r+1}(F|_V)=\Gamma_o \in \CH_n(X^3\backslash\delta_X)$, where
$\Gamma_o$ is the restriction of the variety $\Gamma$ constructed in
(\ref{Gamma}) to the open subset $X^3\backslash\delta_X$.
\end{lemma}
\begin{proof}
By construction, $\Gamma_o$ is exactly the zero locus of the section
$s_f$ of $F|_V$. By Lemma \ref{F(X)}, $\Gamma_o$ is $n$-dimensional,
thus represents the top Chern class of $F|_V$.
\end{proof}

Now consider the cartesian diagram (Lemma \ref{intersectionlemma}):
\begin{equation}\label{diagram}
\xymatrix{
V\cup \Delta_{12}\cup\Delta_{13}\cup\Delta_{23} \cart \ar@{^{(}->}[r]^(.7){i_2} \ar@{^{(}->}[d]_{i_4} & X^3\backslash\delta_X \ar@{^{(}->}[d]^{i_1}\\
W \ar@{^{(}->}[r]_{i_3} &  \P^{\times 3}\backslash\delta_{\P}
}
\end{equation}
Since $i_1$ is clearly a regular embedding, we can apply the theory
of refined Gysin maps of \cite{MR1644323} to the cycle
$c_{n-r+1}(F)\in \CH^{n-r+1}(W)$ in the above diagram. Before doing
so, recall that in Lemma \ref{intersectionlemma} we have observed
that the intersections along $\Delta_{ij}$'s are not proper. Let us
first calculate the excess normal bundles of them.

\begin{lemma}\label{excess}
For any $1\leq i<j\leq 3$, the excess normal sheaf along
$\Delta_{ij}\backslash V$ is a rank $r-1$ vector bundle isomorphic
to a quotient
$\frac{\pr_1^*E|_X}{\pr_1^*\sO_X(1)\otimes\pr_2^*\sO_X(-1)}$, where
we identify $\Delta_{ij}$ with $X\times X\backslash \Delta_X$, and
$\pr_i$ are the natural projections to two factors.
\end{lemma}
\begin{proof}
For simplicity, assume $i=1,j=2$, and write the inclusion $j:
\Delta_{12}=X\times X\backslash\Delta_X\inj X^3\backslash\delta_X$,
which sends $(x,x')$ to $(x,x,x')$, where $x\neq x'$. Now we are in
the following situation:
\begin{displaymath}
  \xymatrix{
  X\times X\backslash\Delta_X \ar[r]^(.5){j}
  \ar[d] &X^3\backslash\delta_X\ar[dd]^{i_1}\\
  \P^{\times 2}\backslash\Delta_{\P} \ar[d]_{i_3'} \ar[dr]^{i_3''}\\
  W\ar[r]_{i_3} & \P^{\times 3}\backslash\delta_{\P}
  }
\end{displaymath}
The normal bundle of $j$ is obviously $\pr_1^*TX$. And the normal
bundle of $i_3$ sits in the exact sequence:
$$0\to N_{i_3'}\to N_{i_3''}\to N_{i_3}\to 0.$$
The normal bundle of $i_3''$ is $\pr_1^*T\P$. As for the normal
bundle of $i_3'$, let us reinterpret $i_3'$ as:
\begin{displaymath}
  \xymatrix{
  L\times_G L\backslash \Delta_L\ar[rr]^{i_3'} \ar[dr] & & L\times_G L\times_G L\backslash\delta_L \ar[dl]\\
  &G&
  }
\end{displaymath}
where $L\to G$ is the universal $\P^1$-fibration over the
Grassmannian of projective lines $G$. From this we see that the
normal bundle of $i_3'$ is the same as the quotient of the two
relative (over $G$) tangent sheaves, thus the fibre of $N_{i_3'}$ at
$(y,y')\in \P^{\times 2}\backslash\Delta_{\P}$ is canonically
isomorphic to $T_{y}\P^1_{yy'}$. Therefore at the point $(x,x')\in
X\times X\backslash\Delta_X$, the fibre of the excess normal bundle
is canonically isomorphic to
$$\frac{N_{i_3'',(x,x')}}{N_{i_3',(x,x')}+N_{j,(x,x')}}=\frac{T_x\P}{T_x\P^1_{xx'}+T_xX}.$$
As long as the line $\P^1_{xx'}$ is not tangent to $X$ at $x$, \ie
$(x,x')\notin V$, the sum in the denominator is a direct sum, and
the fibre of the excess bundle at this point is canonically
isomorphic to the $(r-1)$-dimensional vector space
$$\frac{N_{X/\P,x}}{T_x\P^1_{xx'}}=\frac{E_x}{\Hom_{\C}(\C
\dot{x},\C \dot{x'})},$$ where $\C\dot{x}$ is the 1-dimensional
sub-vector space corresponding to $x\in \P$. Therefore along
$\Delta_{ij}\backslash V$, the excess normal bundle is isomorphic to
$\frac{\pr_1^*E|_X}{\pr_1^*\sO_X(1)\otimes\pr_2^*\sO_X(-1)}$.
\end{proof}

Now we consider the Gysin map $i_1^!$ in the diagram (\ref{diagram})
to get the following.

\begin{prop}\label{basicequality}
There exists a symmetric homogeneous polynomial $P$ of degree $2n$
with integer coefficients, such that in
$\CH_n(X^3\backslash\delta_X)$,
\begin{equation} \label{basiceq}
c_{n-r+1}(F|_{V})+j_{12*}(\alpha)+j_{13*}(\alpha)+j_{23*}(\alpha)+P(h_1,h_2,h_3)=0,
\end{equation}
where $h_i=\pr_i^*(h)\in \CH^1(X^3\backslash\delta_X)$ with
$h=c_1(\sO_X(1))$, $i=1,2,3$; the cycle $\alpha$ is defined by
\begin{equation}\label{alpha}
\alpha=c_{n-r+1}\left(F|_{\Delta_{12}}\right)\cdot
c_{r-1}\left(\frac{\pr_1^*E|_X}{\pr_1^*\sO_X(1)\otimes\pr_2^*\sO_X(-1)}\right)
\in \CH_n(X^2\backslash\Delta_X);
\end{equation}
and the morphisms $j_{12}, j_{13}, j_{23}: X^2\backslash\Delta_X \inj X^3\backslash\delta_X$ are defined by
\begin{eqnarray*}
  j_{12}: &(x,x')&\mapsto  (x,x,x');\\
  j_{13}: &(x,x')&\mapsto  (x,x',x);\\
  j_{23}: &(x,x')&\mapsto  (x',x,x).\\
\end{eqnarray*}
\end{prop}
\begin{proof}
  By the commutativity of Gysin map and push-forwards (\cite{MR1644323} Theorem 6.2(a)):
\begin{equation}\label{equation1}
i_{2*}\left(i_1^{!}c_{n-r+1}(F)\right)=i_1^{!}\bigg(i_{3*}c_{n-r+1}(F)\bigg)
~~\text{in}~ \CH_n(X^3\backslash\delta_X).
\end{equation}
While in the right hand side, $i_{3*}c_{n-r+1}(F)\in
\CH_{n+3r}(\P^{\times 3}\backslash\delta_{\P})\isom
\CH_{n+3r}(\P^{\times 3})$, and the Chow ring of $\P^{\times 3}$ is
well-known: $$\CH^*(\P^{\times 3})=\Z[H_1,H_2,H_3]/(H_i^{n+r+1};
i=1,2,3),$$ where $H_i=\pr_i^*(H)$ with $H\in \CH^1(\P)$ being the
hyperplane section class. Hence there exists a symmetric homogeneous
polynomial $P$ of degree $2n$ with integer coefficients, such that
$$i_{3*}c_{n-r+1}(F)=-P(H_1, H_2, H_3)  ~~\text{in}~
\CH_{n+3r}(\P^{\times 3}\backslash\delta_{\P}).$$ Combining this
with (\ref{equation1}), and denoting $h_i=H_i|_{X^3}\in \CH^1(X^3)$,
we obtain the following equality
\begin{equation}\label{equation2}
 i_{2*}\left(i_1^{!}c_{n-r+1}(F)\right)+P(h_1, h_2, h_3)=0  ~~\text{in}~ \CH_{n}(X^3\backslash\delta_X).
\end{equation}
In the left hand side, by \cite{MR1644323} Proposition 6.3, we have:
\begin{equation}\label{equation3}
i_1^{!}c_{n-r+1}(F)=i_1^{!}\left(c_{n-r+1}(F)\cdot
[W]\right)=c_{n-r+1}(F|_{V\cup\Delta_{12}\cup\Delta_{13}\cup\Delta_{23}})\cdot
i_1^!([W]),
\end{equation}
where $[W]$ is the fundamental class of $W$. Note here $i_1^!([W])$
is a $(2n-r+1)$-dimensional cycle, but $V\cap \Delta_{ij}$ is of
dimension strictly less than $2n-r+1$, thus we can use the excess
intersection formula (\cite{MR1644323} \S 6.3) component by
component in the open subsets $\Delta_{ij}\backslash V$ to get (the
excess normal bundle is given in Lemma \ref{excess}):
$$i_1^!([W])=[V]+\sum_{1\leq i<j\leq 3}[\Delta_{ij}]\cdot
c_{r-1}\left(\frac{\pr_1^*E|_X}{\pr_1^*\sO_X(1)\otimes\pr_2^*\sO_X(-1)}\right).$$
Therefore, by omitting all the push-forwards induced by inclusions
of subvarieties of $X^3\backslash\delta_X$,
$$i_{2*}\left(c_{n-r+1}(F|_V)\cdot i_1^!([W])\right)=c_{n-r+1}(F|_V);$$
$$i_{2*}\left(c_{n-r+1}(F|_{\Delta_{12}})\cdot i_1^!([W])\right)=c_{n-r+1}(F|_{\Delta_{12}})\cdot c_{r-1}\left(\frac{\pr_1^*E|_X}{\pr_1^*\sO_X(1)\otimes\pr_2^*\sO_X(-1)}\right),$$
putting these in (\ref{equation2}) and (\ref{equation3}) we get the
desired formula.
\end{proof}

Let us now deal with the equality (\ref{basiceq}) term by term. Firstly,  by Lemma \ref{gamma2}, $c_{n-r+1}(F|_V)=\Gamma_o$ in $\CH_n(X^3\backslash\delta_X)$. Secondly, we would like to calculate $F|_{\Delta_{12}}$ of
Proposition \ref{basicequality}. We remark that this bundle is the
pull-back by the inclusion $X^2\backslash\Delta_X\inj \P^{\times
2}\backslash\Delta_{\P}$ of the bundle $$M:= F|_{\Delta_{12,\P}}$$
where $\Delta_{12,\P}=\left\{(y,y,y')\in \P^{\times 3}~|~~ y\neq
y'\right\}\subset W$, and we identify $\Delta_{12,\P}$ with
$\P^{\times 2}\backslash\Delta_{\P}$.

We still use $S$ to denote the tautological rank 2 vector bundle on
$\P^{\times 2}\backslash\Delta_{\P}$, hence $p: \P(S)\to \P^{\times
2}\backslash\Delta_{\P}$ is the universal line, which admits two
tautological sections $\sigma,\sigma'$, and we call $q: \P(S)\to \P$
the natural morphism:
\begin{equation}\label{diagram2}
 \xymatrix{
\P(S) \ar[r]^{q} \ar[d]^{p} & \P\\
\P^{\times 2}\backslash\Delta_{\P}\ar@/^/[u]<3ex>|{\sigma}
\ar@/^/[u]<1ex>|{\sigma'} }
\end{equation}
\begin{lemma}\label{M}
Notations as in the diagram (\ref{diagram2}) above, then $M\isom
\left(\sO_{\P}(1)\boxtimes\sO_{\P}(2)\right)\otimes
p_*\left(q^*E(-3)\right).$
\end{lemma}
\begin{proof}
By construction (or see (\ref{F})),
\begin{equation}\label{equation4}
M=p_*\left(q^*E\otimes \sO_{\P(S)}(-2D-D')\right),
\end{equation}where $D$,
$D'$ is the images of the sections $\sigma, \sigma'$. Since $p$ is a
projective bundle and the intersection number of $D$ with the fibre
is 1, we can assume that
$\sO_{\P(S)}(-D)=p^*\left(\sO_{\P}(a)\boxtimes\sO_{\P}(b)\right)\otimes
\sO_{\P(S)}(-1)$. Pushing forward by $p_*$ the exact sequence $$0\to
\sO_{\P(S)}(-D)\otimes\sO_{\P(S)}(1)\to \sO_{\P(S)}(1) \to
\sO_{\P(S)}(1)|_D\to 0,$$ we find an exact sequence:
$$0\to \sO_{\P}(a)\boxtimes\sO_{\P}(b)\to S\dual\to \sO_{\P}(1)\boxtimes\sO_{\P}\to 0,$$
where the last term comes from the fact that
$p_*\left(\sO_{\P(S)}(1)|_D\right)=\sigma^*(\sO_{\P(S)}(1))$ whose
fibre at $(y,y')$ is $(\C\dot{y})^*$. Now noting
$S\dual=\pr_1^*\sO_{\P}(1)\oplus \pr_2^*\sO_{\P}(1)$, and
restricting to $\P\times\{\pt\}$ and $\{\pt\}\times \P$, we get
$a=0, b=1$, \ie
$$\sO_{\P(S)}(-D)=p^*\left(\pr_2^*\sO_{\P}(1)\right)\otimes
\sO_{\P(S)}(-1).$$ Similarly,
$\sO_{\P(S)}(-D')=p^*\left(\pr_1^*\sO_{\P}(1)\right)\otimes
\sO_{\P(S)}(-1)$. Putting these into (\ref{equation4}), the
projection formula finishes the proof of Lemma.
\end{proof}

Combining Proposition \ref{basicequality}, Lemma \ref{gamma2} and
Lemma \ref{M}, we have

\begin{prop}\label{predecompProp}
In $\CH_n(X^3\backslash\delta_X)$, we have
\begin{equation}\label{predecomp}
\Gamma_o+j_{12*}(Q(h_1,h_2))+j_{13*}(Q(h_1,h_2))+j_{23*}(Q(h_1,h_2))+P(h_1,h_2,h_3)=0,
\end{equation}
where $P$ is a symmetric homogeneous polynomial of degree $2n$ in
three variables with integer coefficients; $h_i\in
\CH^1(X^3\backslash\delta_X)$ or $\CH^1(X^2\backslash\Delta_X)$ is
the pull-back of $h=c_1(\sO_X(1))$ by the $i^ \text{th}$ projection;
the inclusions $j_{12}, j_{13}, j_{23}: X^2\backslash\Delta_X \inj
X^3\backslash\delta_X$ are defined as in Proposition
\ref{basicequality}; and $Q$ is a homogeneous polynomial of degree
$n$ in two variables with integer coefficients determined
by\footnote{Here
$\frac{\pr_1^*E}{\pr_1^*\sO_{\P}(1)\otimes\pr_2^*\sO_{\P}(-1)}$ is
no more a quotient vector bundle, but only an element in the
Grothendieck group of vector bundles on $\P^{\times 2}$ on which the
Chern classes are however still well-defined.}:
\begin{equation}\label{Q}
Q(H_1,H_2)=c_{n-r+1}(M)\cdot
c_{r-1}\left(\frac{\pr_1^*E}{\pr_1^*\sO_{\P}(1)\otimes\pr_2^*\sO_{\P}(-1)}\right)
\in \CH^n(\P^{\times 2}\backslash\Delta_{\P})\isom \CH^n(\P^{\times
2});
\end{equation}
where $M\isom \left(\sO_{\P}(1)\boxtimes\sO_{\P}(2)\right)\otimes
p_*\left(q^*E(-3)\right)$ as in Lemma \ref{M}.
\end{prop}

By the localization exact sequence
$\CH_n(X)\lra{\delta_*}\CH_n(X^3)\to \CH_n(X^3\backslash\delta_X)\to
0$, we deduce from (\ref{predecomp}) the following decomposition of
small diagonal:

\begin{cor}\label{decomp}
Let $P$, $Q$ be the same polynomials as in Proposition \ref{predecompProp}.
Then there exists an integer $N$, such that we have a decomposition
of $N\cdot\delta_X$ in $\CH_n(X^3)$:
\begin{equation}\label{decompeqn}
N\cdot\delta_X=
\Gamma+j_{12*}(Q(h_1,h_2))+j_{13*}(Q(h_1,h_2))+j_{23*}(Q(h_1,h_2))+P(h_1,h_2,h_3),
\end{equation}
where we still denote by $j_{12}, j_{23}, j_{13}: X^2\inj X^3$ the
inclusions defined by the same formulae as before, and recall that
$\Gamma$ is the subvariety constructed in (\ref{Gamma}), which is
the closure of $\Gamma_o$ in $X^3$.
\end{cor}

In the equation of Corollary \ref{decomp} above, we make the
following observation of relations between $N$ and the coefficients
of $P$ and $Q$ by using the \emph{non-existence} of decomposition of
diagonal $\Delta_X\subset X\times X$ in the sense of Bloch-Srinivas
(see Theorem \ref{BS} of the introduction) for smooth projective
varieties with $H^{n,0}\neq 0$, for example varieties of Calabi-Yau
type.

The \emph{degree} of $X$ is given by the maximal self-intersection
number of the hyperplane section: $\deg(X)=\bigg(\underbrace{h\cdot
h\cdots h}_n\bigg)_X$, which is in fact the top Chern number of $E$.

\begin{lemma}\label{coefflemma}
   Write the $\Z$-coefficient polynomials $$P(H_1, H_2, H_3)=\sum_{\substack{i+j+k=2n\\i,j,k\geq 0}}b_{ijk}H_1^iH_2^jH_3^k ~~~\text{with}~ b_{ijk}~ \text{symmetric on the indices};$$
   $$Q(H_1,H_2)=a_nH_1^n+a_{n-1}H_1^{n-1}H_2+\cdots+a_0H_2^n.$$
   Then we have
   \begin{equation}\label{coeff1}
       N=a_0\cdot\deg(X);
   \end{equation}
   \begin{equation}\label{coeff2}
       a_i+a_j=-b_{ijn}\cdot\deg(X)~~~\text{for any}~ i+j=n.
   \end{equation}
\end{lemma}
\begin{proof}
Applying to the equation in Corollary \ref{decomp} the push-forward
induced by the projection to the first two factors $\pr_{12}:
X^3\surj X^2$, then
\begin{itemize}
\item $\pr_{12*}(N\cdot\delta_X)=N\cdot\Delta_X;$
\item $\pr_{12 *}(\Gamma)=0$, since $\Gamma$ has relative dimension 1
for $\pr_{12}$;
\item $\pr_{12*}\circ j_{12}\left(Q(h_1,h_2)\right)=\Delta_*\circ \pr_{1*}\left(Q(h_1,h_2)\right)=a_0\deg(X)\cdot\Delta_X;$
\item $\pr_{12*}\circ j_{13}\left(Q(h_1,h_2)\right)=\id_{X^2*}\left(Q(h_1,h_2)\right)=Q(h_1,h_2);$
\item $\pr_{12*}\circ
j_{23}\left(Q(h_1,h_2)\right)=\iota_{X^2*}\left(Q(h_1,h_2)\right)=Q(h_2,h_1)$,
where $\iota:X^2\to X^2$ is the the involution interchanging the two
coordinates;
\item $\pr_{12
*}\left(P(h_1,h_2,h_3)\right)=\deg(X)\cdot\sum_{i+j=n}b_{ijn}h_1^ih_2^j$.
\end{itemize}
Putting these together, we obtain that in $\CH_n(X^2)$,
\begin{equation}\label{AbsurdDecomp}
\left(N-a_0\cdot\deg(X)\right)\cdot\Delta_X=\sum_{i+j=n}(b_{ijn}\deg(X)+a_i+a_j)\cdot
h_1^ih_2^j.
\end{equation}
If $N-a_0\cdot\deg(X)\neq 0$, then (\ref{AbsurdDecomp}) gives a
nontrivial decomposition of diagonal $\Delta_X\in \CH_n(X^2)$ of
Bloch-Srinivas type, but this is impossible: we regard
(\ref{AbsurdDecomp}) as an equality of cohomological correspondences
from $H^n(X)$ to itself, then the left hand side acts by multiplying
a non-zero constant $\left(N-a_0\cdot\deg(X)\right)$, while the
action of the right hand side has image a sub-Hodge structure of
coniveau at least 1, which contradicts to the non-vanishing of
$H^{n,0}(X)$.\\
As a result, we have (\ref{coeff1}), and hence (\ref{coeff2}) since
$\left\{h_1^ih_2^j\right\}_{i+j=n}$ are linearly independent.
\end{proof}

Using Lemma \ref{coefflemma}, we obtain improved version of
Corollary \ref{decomp} as summarized in the following

\begin{thm}\label{main1}
  Let $\P, E, X$ be as in the basic setting. Then there is a
  homogeneous polynomial with $\Z$-coefficients
  $$Q(H_1,H_2)=a_nH_1^n+a_{n-1}H_1^{n-1}H_2+\cdots+a_0H_2^n,$$
  which is determined by (\ref{Q}), such that in $\CH_n(X^3)$
  \begin{equation}\label{main1eqn}
a_0\deg(X)\cdot\delta_X=
\Gamma+j_{12*}(Q(h_1,h_2))+j_{13*}(Q(h_1,h_2))+j_{23*}(Q(h_1,h_2))+P(h_1,h_2,h_3),
  \end{equation}
where $P$ is the symmetric  polynomial with $\Z$-coefficients $$P(H_1, H_2,
H_3)=\sum_{i+j+k=2n}b_{ijk}H_1^iH_2^jH_3^k,$$ with (\ref{coeff2}):
$a_i+a_j=-b_{ijn}\cdot\deg(X)$ for any $i+j=n$; and $h_i\in
\CH^1(X^3)$ or $\CH^1(X^2)$ is the pull-back of $h=c_1(\sO_X(1))$ by
the $i^{th}$ projection, and the inclusions of big diagonals are
given by:
\begin{eqnarray*}
  & X^2 &\inj X^3\\
  j_{12}: &(x,x')&\mapsto  (x,x,x')\\
  j_{13}: &(x,x')&\mapsto  (x,x',x)\\
  j_{23}: &(x,x')&\mapsto  (x',x,x).\\
\end{eqnarray*}
\end{thm}
This theorem generalizes the result of Voisin \cite{VoiFamilyK3} for
Calabi-Yau hypersurfaces (see Theorem \ref{V} of the introduction)
except for a small point: to get a non-trivial decomposition of the
small diagonal and thus applications like Corollary \ref{VCor}, we
need to verify that $a_0\neq 0$. It is the case when $E$ is
splitting, \ie when $X$ is Calabi-Yau complete intersection, see the following subsections.

\subsection{Applications to the multiplicative structure of Chow rings}

Always in the same setting as the previous subsection. Now we can regard (\ref{main1eqn}) as an equality of correspondences
from $X\times X$ to $X$. Combining with (\ref{coeff2}), we
get the following corollary in the same spirit of Corollary
\ref{VCor} of the introduction.

\begin{cor}\label{corweak}
 In the same notation as before.
   Let $Z\in \CH^k(X), Z'\in \CH^l(X)$ be two algebraic cycles of
   $X$ of codimension $k,l\in \N$ with $k+l=n$. Then we have an equality in
   $\CH_0(X)$, here $\bullet$ means the intersection product in
   $\CH^*(X)$:
   $$a_0\deg(X)\cdot Z\bullet Z'=a_0\deg(Z\bullet Z')\cdot h^n +a_l\deg(Z')\cdot Z\bullet
   h^l+ a_k\deg(Z)\cdot Z'\bullet h^k- (a_k+a_l)\frac{\deg(Z)\deg(Z')}{\deg(X)}\cdot h^n,$$
   where the \emph{degree} of an algebraic cycle $\sZ$ is defined to be the intersection number $(\sZ\cdot h^{\dim
   \sZ})_X$.
\end{cor}
\begin{proof} Let us list term by term the results of applying the
correspondences in (\ref{main1eqn}) to the cycle $Z\times Z'\in
\CH^{n}(X\times X)$:
\begin{itemize}
  \item $\left(a_0\deg(X)\cdot\delta_X\right)_*(Z\times Z')=a_0\deg(X)\cdot Z\bullet Z';$
  \item $\Gamma_*(Z\times Z')=0$, since $\Gamma_*(Z\times Z')$ is represented
  by a linear combination of some fundamental class of a subvariety of dimension at least 1,
  but $\Gamma_*(Z\times Z')$ should be a zero-dimensional cycle, so it vanishes;
  \item $\left(j_{12*}(h_1^ih_2^{n-i})\right)_*(Z\times Z')=(Z\cdot Z'\cdot h^i)_X\cdot
  h^{n}$ if $i=0$, and vanishes otherwise, therefore
  $$\bigg(j_{12*}Q(h_1,h_2)\bigg)_*(Z\times Z')=a_0\deg(Z\bullet Z')\cdot
  h^{n};$$
  \item $\left(j_{13*}(h_1^ih_2^{n-i})\right)_*(Z\times Z')=(Z'\cdot h^{n-l})_X\cdot
  Z\bullet h^l$ if $i=l$, and vanishes otherwise, therefore
  $$\bigg(j_{13*}Q(h_1,h_2)\bigg)_*(Z\times Z')=a_l\deg(Z')\cdot
  Z\bullet h^l;$$
  \item $\left(j_{23*}(h_1^ih_2^{n-i})\right)_*(Z\times Z')=(Z\cdot h^{n-k})_X\cdot
  Z'\bullet h^k$ if $i=k$, and vanishes otherwise, therefore
  $$\bigg(j_{23*}Q(h_1,h_2)\bigg)_*(Z\times Z')=a_k\deg(Z)\cdot
  Z'\bullet h^k;$$
  \item $(h_1^{i_1}h_2^{i_2}h_3^{i_3})_*(Z\times Z')=(h^{l}\cdot Z)_X\cdot(h^{k}\cdot
  Z')_X\cdot h^{n}$ if $(i_1, i_2, i_3)=(l, k, n)$, and vanishes
  otherwise, therefore $$\bigg(P(h_1,h_2,h_3)\bigg)_*(Z\times Z')=b_{l,k,n}\deg(Z)\deg(Z')\cdot h^{n}.$$
\end{itemize}
Putting all these together, we deduce exactly what we
want.
\end{proof}

% By taking $Z=h^k$, $Z'=h^l$ in the above corollary, we
%calculate all the coefficients of the polynomial $P$ in terms of the
%coefficients of $Q$, generalizing (\ref{coeff2}):
%
%\begin{lemma}\label{coeff3lemma}
%For any three integers $k,l,m\in \N$ with $k+l+m=n$, we have:
%\begin{equation}\label{coeff3}
%b_{n-k,n-l,n-m}\cdot\deg(X)=a_0-a_k-a_l-a_m.
%\end{equation}
%\end{lemma}
%
%
%
%For the ease of later reference, we remark that in the presence of
%(\ref{coeff3}), Corollary \ref{corweak} reads:
%\begin{prop}\label{middlestep}
%For any $Z\in \CH^k(X), Z'\in \CH^l(X)$ with $m:=n-k-l\geq 0$, then
%in $\CH^{k+l}(X)$ we have:
%\begin{eqnarray}\label{middle}
%\lefteqn{ a_0\deg(X)\cdot Z\bullet Z'=a_m\deg(Z\bullet Z')\cdot
%h^{k+l}
%+a_l\deg(Z')\cdot Z\bullet h^l} \nonumber\\
%& & {}+ a_k\deg(Z)\cdot Z'\bullet h^k+
%\frac{a_0-a_k-a_l-a_m}{\deg(X)}\deg(Z)\deg(Z')\cdot h^{k+l}.
%\end{eqnarray}
%\end{prop}

% Before we proceed further, let us treat the special case of
% Calabi-Yau complete intersections, which serves as the main interest
% or motivation of this paper as well as an illustration of the
% general results in the third subsection.

To simplify further the equality in Corollary \ref{corweak}, we need:
\begin{lemma}\label{laststep}
  Suppose $a_1\neq a_0$. Let $k\in \{0,1,\cdots,n-1\}$.
  Then for any $\sZ\in \CH^k(X)_{\Q}$, we have $\sZ\bullet h^{n-k}$
  is always proportional to $h^n$ in $\CH_0(X)_{\Q}$.
\end{lemma}
\begin{proof}
Given $\sZ\in \CH^k(X)_{\Q}$, replacing $Z$ by $\sZ\bullet h^{n-k-1}$ and $Z'$ by $h$ in the formula of Corollary \ref{corweak}, we get:
$$(a_0-a_1)\deg(X)\cdot \left(Z\bullet h^{n-k}-\frac{\deg(\sZ)}{\deg(X)}h^n\right)=0.$$
Since $a_0\neq a_1$, we can divide out $(a_0-a_1)\deg(X)$, obtaining
$Z\bullet h^{n-k}=\frac{\deg(Z)}{\deg(X)}h^n$.
\end{proof}

Inserting this lemma into the formula of Corollary \ref{corweak},
its last three terms simplify with each other, and we finally obtain
our main consequence of the decomposition theorem:
\begin{thm}\label{main2}
Let $E$ be a rank $r$ vector bundle on $\P^{n+r}$ satisfying the positivity condition $(*)$ as well as the Calabi-Yau condition: $\det(E) \isom \sO_{\P}(n+r+1)$. Let $a_i$ still be the coefficients determined by (\ref{Q}). Suppose $a_0\neq 0$ and $a_1\neq a_0$. Let $X$ be the (Calabi-Yau) zero locus of a general section of $E$. Then for any strictly positive
integers $k,l\in \N^*$, with $k+l=n$.
$$\im\left(\bullet: \CH^k(X)_{\Q}\times \CH^l(X)_{\Q}\to \CH_0(X)_{\Q}\right)=\Q\cdot h^n,$$
where $h=c_1(\sO_X(1))\in \CH^1(X)$.
\end{thm}

\subsection{Splitting case: Calabi-Yau complete intersections}
In this subsection, we deal with the special case that $E$ is of
splitting type:
$$E= \bigoplus_{i=1}^r\sO_{\P}(d_i),$$ where $d_1\geq d_2\geq\cdots\geq d_r\geq
2$, with $d=\sum_{i=1}^rd_i=n+r+1$. Hence $X$ is a smooth Calabi-Yau
complete intersection of multidegree $(d_1,\cdots,d_r)$. In
particular, $\deg(X)=\prod_{i=1}^rd_i$.

We are going to prove that the conditions $a_0\neq 0$ and $a_1\neq a_0$ appeared in Theorem \ref{main2} are satisfied in this special case by the following calculations:

\begin{lemma}\label{spli}
  If $E$ is of splitting type as above, then
  \begin{enumerate}
    \item The vector bundle $M$ in Lemma \ref{M} is isomorphic to the restriction to $\P^{\times 2}\backslash\Delta_{\P}$ of
    $$\bigoplus_{i=1}^r\bigoplus_{j=1}^{d_i-2}\sO_{\P}(j)\boxtimes\sO_{\P}(d_i-j).$$
    In particular in (\ref{Q}),
    \begin{equation}\label{c(M)split}
c_{n-r+1}(M)=\prod_{i=1}^r\prod_{j=1}^{d_i-2}(jH_1+(d_i-j)H_2).
    \end{equation}
    \item In (\ref{Q}),
    $c_{r-1}\left(\frac{\pr_1^*E}{\pr_1^*\sO_{\P}(1)\otimes\pr_2^*\sO_{\P}(-1)}\right)$
    is the degree $(r-1)$ part of the formal series $$\frac{\prod_{i=1}^r(1+d_iH_1)}{1-(H_2-H_1)}.$$
    \item The coefficient of $H_2^n$ in the polynomial $Q$ is given
    by $a_0=\prod_{i=1}^r\left((d_i-1)!\right)$, which is non-zero in particular.
    \item The coefficient of $H_1H_2^{n-1}$ in the polynomial $Q$ is
    given by $$a_1=\left(\prod_{i=1}^r(d_i-1)!\right)\cdot\left(\left(\sum_{i=1}^r\sum_{j=1}^{d_i-2}\frac{j}{d_i-j}\right)+n+2\right),$$
    in particular $a_1\neq a_0$.
  \end{enumerate}
\end{lemma}
\begin{proof}
    In the situation as in diagram (\ref{diagram2}),
    \begin{eqnarray*}
     p_*\left(q^*E(-3)\right)=&\oplus_{i=1}^rp_*q^*\sO_{\P}(d_i-3)
     &\\
     =&\oplus_{i=1}^rp_*\sO_{\P(S)}(d_i-3)& (\text{since}~ q^*\sO_{\P}(1)=\sO_{\P(S)}(1))\\
     =&\oplus_{i=1}^r\Sym^{d_i-3}S\dual  &(d_i-3\geq -1, \text{define}~ \Sym^0=\sO,
     \Sym^{-1}=0)\\
     =&\oplus_{i=1}^r\oplus_{j=0}^{d_i-3}\left(\sO_{\P}(j)\boxtimes\sO_{\P}(d_i-3-j)\right)
     &(\text{recall}~S\isom \pr_1^*\sO_{\P}(-1)\oplus \pr_2^*\sO_{\P}(-1))\\
    \end{eqnarray*}
    Thus $M\isom\left(\sO_{\P}(1)\boxtimes\sO_{\P}(2)\right)\otimes
p_*\left(q^*E(-3)\right)=
\bigoplus_{i=1}^r\bigoplus_{j=1}^{d_i-2}\left(\sO_{\P}(j)\boxtimes\sO_{\P}(d_i-j)\right)$,
and the top Chern class follows immediately.\\
The second point is obvious. As for the coefficient $a_0$, by
(\ref{Q}) it is the product of the coefficient of $H_2^{n-r+1}$ in
$c_{n-r+1}(M)$ and the coefficient of $H_2^{r-1}$ in
$c_{r-1}\left(\frac{\pr_1^*E}{\pr_1^*\sO_{\P}(1)\otimes\pr_2^*\sO_{\P}(-1)}\right)$,
which are $\prod_{i=1}^r(d_i-1)!$ and 1 respectively, by the first
two parts of this lemma. Remembering
$\left(\sum_{i=1}^rd_i\right)-r+1=n+2$, the calculation for $a_1$ is also straightforward.
\end{proof}

As a result, in this complete intersection case the decomposition Theorem \ref{main1} and its application Theorem \ref{main2} on the multiplicative structure on the Chow rings  then read as following respectively:

\begin{thm}\label{main1CI}
Let $X$ be a general Calabi-Yau complete intersection in a projective space, then in $\CH_n(X^3)$ we have a decomposition of the small diagonal:
\begin{equation*}
\left(\prod_{i=1}^r(d_i!)\right)\cdot
\delta_X=\Gamma+j_{12*}(Q(h_1,h_2))+j_{13*}(Q(h_1,h_2))+j_{23*}(Q(h_1,h_2))+P(h_1,h_2,h_3),
\end{equation*}
\end{thm}

\begin{thm}\label{main2CI}
Let $X$ be a general Calabi-Yau complete intersection in a projective space, then for any strictly positive
integers $k,l\in \N^*$, with $k+l=n$.
$$\im\left(\bullet: \CH^k(X)_{\Q}\times \CH^l(X)_{\Q}\to \CH_0(X)_{\Q}\right)=\Q\cdot h^n,$$
where $h=c_1(\sO_X(1))\in \CH^1(X)$.
\end{thm}

\begin{rmk}\label{BB}\upshape
Theorem \ref{main2CI} can be reformulated as: for a general Calabi-Yau complete intersection in a projective space, any \emph{decomposable} 0-cycle with $\Q$-coefficient is $\Q$-rational equivalent to zero if and only if it has degree 0. Here a $0$-cycle is called \emph{decomposable} if it is in the sum
of the images: $$\sum_{\substack{k+l=n\\
k,l>0}}\im\left(\CH^k(X)_\Q\times\CH^l(X)_\Q\lra{\bullet}\CH_0(X)_\Q\right).$$
It is interesting to compare this result for Calabi-Yau varieties with Beauville's \emph{`weak splitting principle'} conjecture in the holomorphic symplectic case. In \cite{MR2047674}, Beauville and Voisin reinterpreted their result (see Corollary \ref{BVCor} in \S0 Introduction) as some sort of compatibility with splitting of the conjectural Bloch-Beilinson filtration. In \cite{MR2187148}, Beauville proposed (and checked some examples of) a weak form of such compatibility for higher dimensional irreducible holomorphic symplectic varieties to check, namely his `weak splitting property' conjecture. Later in \cite{MR2435839}, Voisin formulates a stronger version of this conjecture: \emph{for an irreducible holomorphic symplectic projective variety, any polynomial relation between the \emph{cohomological} Chern classes of lines bundles and the tangent
bundle holds already for their \emph{Chow-theoretical} Chern
classes}. She also proved this conjecture for the variety of lines in a cubic four-fold and for Hilbert schemes of points on K3 surfaces in certain range (\cf \cite{MR2435839}).

Note that the Chern classes of the tangent bundle of a complete intersection is given by $$c(T_X)=\left(\frac{c(T_{\P})}{c(E)}\right)\bigg|_X,$$ which is clearly a cycle coming from the ambient projective space. Therefore comparing to the `weak splitting principle' for holomorphic symplectic varieties, our result for Calabi-Yau varieties is on one hand stronger, in the sense that besides the divisors and Chern classes of the tangent bundle, cycles of all strictly positive codimension are allowed in the polynomial; and on the other hand weaker since only the polynomials of (weighted) degree $n$ are taken into account.
\end{rmk}

%Let us bring Proposition \ref{middlestep} into consideration again.
%Taking $Z'=h$ in (\ref{middle}), the formula simplifies as: for any
%$Z\in \CH^k(X)$ with $0\leq k<n$,
%$$(a_0-a_1)(\deg(X)\cdot Z-\deg(Z)\cdot h^k)\bullet h=0~~\text{in}~ \CH^{k+1}(X).$$
%Since $a_0\neq a_1$, we have the following corollary, which is
%somehow interesting in its own right.
%
%\begin{cor}
%Let $X$ be an $n$-dimensional general smooth Calabi-Yau complete
%intersection in projective space. Then for any $0\leq k\leq n$,
%$$\im\left(\bullet h: \CH^k(X)_{\Q}\to \CH^{k+1}(X)_{\Q}\right)=\Q\cdot h^{k+1},$$
%where $h=c_1(\sO_X(1))\in \CH^1(X)$.
%\end{cor}
%
%Now we give another look at Proposition \ref{middlestep}, thanks to
%the preceding lemma, using \Q-coefficient Chow groups, we can
%replace $Z\bullet h^l$ by $\frac{\deg(Z)}{\deg(X)}h^{k+l}$ and
%$Z'\bullet h^k$ by $\frac{\deg(Z')}{\deg(X)}h^{k+l}$ as long as
%$k,l\geq 1$. Therefore the formula (\ref{middle}) therein simplifies
%tremendously, and we obtain the following surprising consequence,
%which roughly says that the intersection products of $\CH^*_{\Q}$ for
%Calabi-Yau complete intersections is totally `degenerate'.

\begin{rmk}[non-splitting case]\label{nonsplit}\upshape
In this remark, we discuss the possibility to generalize the above results to the broadest case of $X$ being the zero locus of a general section of vector bundle satisfying the positivity condition $(*)$ and Calabi-Yau condition $\det(E)\isom \sO_{\P}(n+r+1)$. By Theorem \ref{main2}, to draw the same conclusion that any degree 0 decomposable $\Q$-0-cycle is rationally equivalent to 0 (\cf Remark \ref{BB}), the only thing we need is to verify that the coefficients of the polynomial $Q$ defined in (\ref{Q}) satisfy the conditions\\
($\lozenge\lozenge$): $a_0\neq 0$ and $a_1\neq a_0$.\\
Although no proof is found for the time being, the author feels that very probably ($\lozenge\lozenge$) is always true under the condition $(*)$. In fact, once the Chern classes of $E$ are given, the coefficients $a_0$ and $a_1$ are practically computable:
Grothendieck-Riemann-Roch formula provides us the Chern character of the vector bundle $M$ in Lemma \ref{M}:
\begin{equation}
\ch(M)(t)=e^{(H_1+2H_2)t}\cdot\sum_{i=1}^r\frac{e^{(d_i-2)H_1t}-e^{(d_i-2)H_2t}}{e^{H_1t}-e^{H_2t}},
\end{equation}
from which $c_{n-r+1}(M)$ and thus
 $$Q(H_1,H_2)=c_{n-r+1}(M)\cdot
c_{r-1}\left(\frac{\pr_1^*E}{\pr_1^*\sO_{\P}(1)\otimes\pr_2^*\sO_{\P}(-1)}\right)$$ could be calculated without essential difficulties in practice. It is clear at this point that, the condition ($\lozenge\lozenge$) is an `open' condition; while the calculation for the splitting case (see the proof of Lemma \ref{spli}) seems to imply that the positivity condition $(*)$ makes ($\lozenge\lozenge$) far from being wrong. It is possible that some theory for the classification of (Chern classes of) vector bundles would be needed to get a genuine proof. A detailed discussion will appear in the author's future thesis.
\end{rmk}

\subsection{A contrasting example} \label{examplesection}
The results of Corollary \ref{BVCor} (\cite{MR2047674}), Corollary
\ref{VCor} (\cite{VoiFamilyK3}), and the generalization Theorem \ref{main2CI} in this paper suggest the following

\paragraph*{Question:} To what extent such degeneracy of intersection
products in Chow ring can be generalized to other smooth projective
varieties?

In this subsection, we construct a contrasting example, showing that in the results above the Calabi-Yau condition is essential, while the complete intersection assumption is not sufficient. More precisely, we will construct a \emph{smooth surface $S$ in $\P^3$ which is of general
type, such that the image of the intersection product map
$$\bullet: \Pic(S)\times\Pic(S)\to \CH_0(S)_{\Q}$$
has some elements not proportional to $h^2$, where
$h=c_1(\sO_S(1))$.}

Let $\Sigma$ be a general smooth quartic surface in $\P^3=:\P$. Thus
$\Sigma$ is a K3 surface. Let us denote by $V_{k,g}$ its
\emph{Severi variety}:
$$V_{k,g}:=\overline{\left\{C\in |\sO_\Sigma(k)|: C ~\text{is irreducible and nodal with}~ g(\tilde C)=g\right\}},$$
where $\tilde C$ means the normalization of $C$, and the closure is
taken in $|\sO_\Sigma(k)|$. One knows that $V_{k,g}$ is
smooth non-empty of dimension $g$ if $0\leq g\leq 2k^2+1$, whose
general member has $\delta:=2k^2+1-g$ nodes (\cf \cite{0557.14015} \cite{MR1675158}).

In particular, $V_{3,2}$ is of dimension 2, and its general member has 17 nodes.
We first show that the nodes of the irreducible curves parameterized by $V_{3,2}$ sweep out a 2-dimensional part of $\Sigma$. Indeed, consider general members $C_1\in V_{1,1}$ and $C_2\in V_{2,1}$, then $C_1$, $C_2$ are both irreducible, of normalization genus 1 and they intersect transversally at 8 points. Note that any morphism from an elliptic curve to K3 surface $f:E\to \Sigma$ with nodal image, the normal bundle of $f$ is $\frac{f^*T_{\Sigma}}{T_E}$, which is clearly trivial. Therefore $C_1$ and $C_2$ both vary in a 1-dimensional family with intersection points running over a 2-dimensional part of $\Sigma$, and thus the unions $C_1\cup C_2$ give a 2-dimensional family of (reducible) curves in $|\sO_{\Sigma}(3)|$ with 18 nodes, where at least one node (in $C_1\cap C_2$) sweeps out a 2-dimensional part of $\Sigma$. Keeping this node, we smoothify another node we get a 2-dimensional family of irreducible curves (with 17 nodes) in $V_{3,2}$ with one node sweeping out a 2-dimensional part as desired.

 As
$\CH_0(\Sigma)_{\Q}$ is different from $\Q$ and generated by the points of $\Sigma$, we can find $C'\in
V_{3,2}$ irreducible with 17 nodes $N_1,\cdots N_{17}$, such that at
least one of its nodes has class in $\CH_0(\Sigma)_{\Q}$ different
from $c_\Sigma:=\frac{1}{4}c_1\left(\sO_\Sigma(1)\right)^2$. Quite
obviously, there exist thus 16 of them, say the first 16, with sum
not rational equivalent to $16c_\Sigma$. Since $V_{3,3}$ is also
smooth, containing $V_{3,2}$ as a smooth divisor, we can deform $C'$
in $|\sO_\Sigma(3)|$ by keeping the first 16 nodes, and smoothifying
the last one, to get $C\in V_{3,3}$ with the sum
$Z:=N_1+\cdots+N_{16}$ of its 16 nodes not rationally equivalent to
$16c_\Sigma$.

Now we take another copy of $\P^3=:\P'$, and construct a finite
cover $\pi: \P'\to \P$ by taking $[X_0:X_1:X_2:X_3]$ to
$[q_0(X):q_1(X):q_2(X):q_3(X)]$ with $q_i$ quadratic polynomials in
$X_0,\cdots, X_3$ without base points, the $q_i$ will be given
later. We want to have an embedding $\tilde C\inj \P'$ such that the
following diagram commutes:
\begin{equation}\label{embed}
  \xymatrix{
  \tilde C\ar@{-->}[r] \ar[d]_{n} & \P'\ar[d]^{\pi}\\
  C\ar[r] & \P
  }
\end{equation}
We consider the square roots of $n^*\sO_C(1)$ in $\Jac(\tilde C)$,
\ie  $L\in \Jac(\tilde C)$ such that $L^{\otimes 2}\isom
n^*\sO_C(1)$. Note that such $L$ is a degree 6 divisor on the genus
3 curve $\tilde C$. We can choose one of these square roots $L$
which is very ample on $\tilde C$. Indeed, if none of them is very
ample, then each square root is of the form $\sO_{\tilde
C}(K_{\tilde C}+x+y)$ for some $x, y\in \tilde C$. Therefore all the
2-torsion points of $\Jac(\tilde C)$ are contained in a translation
of the image of $u: \Sym^2\tilde C\to \Jac(\tilde C)$, which is again a
translation of the theta divisor by Poincar\'e's formula. However,
it is known that for a principally polarized abelian variety, any
translation of a theta divisor cannot contain all the 2-torsion
points, see for example \cite{MR2960840} \cite{MR0204427}
\cite{MR0325625}.

$L$ being chosen as above, the corresponding embedding $i:\tilde
C\inj \P':=\P|L|^*\isom \P^3$, (with $\sO_{\tilde C}(1)=L$), induces
a morphism $i^*: H^0(\P',\sO_{\P'}(2))\to H^0(\tilde C,2L)$. We make
the following\\

\noindent\textbf{Claim$(\bigstar)$}: for the smoothification $C\in
V_{3,3}$
chosen generically, $i^*$ above is an isomorphism.\\

Since both vector spaces have the same dimension 10, we only need to
verify the injectivity of $i^*$, that is, $\tilde C$ in not
contained in a quadric of $\P'$. Suppose on the contrary that there
exists a quadric $Q\subset \P'$ containing $\tilde C$, then we have:

\begin{lemma}
 $\tilde C$ is hyperelliptic.
\end{lemma}
\begin{proof}
If $Q\isom \P^1\times \P^1$ is smooth, we denote the class of its two fibers by $l_1$, $l_2$, and assume $\tilde C=a\cdot l_1+b\cdot l_2$ for $a,b \in \Z$. Since $\tilde C$ is of degree 6 and $\sO_Q(1)=\sO_{\P^1}(1)\boxtimes\sO_{\P^1}(1)$, we have $a+b=6$. On the other hand, the fact that $\tilde C$ is of genus 3 implies that $\left(\tilde C^2\right)+\left((-2l_1-2l_2)\cdot \tilde C\right)=4$, which is equivalent to $2ab-2(a+b)=4$. Therefore $a=2, b=4$ (or $a=4, b=2$), \ie $\tilde C\in \big|\sO_{\P^1}(2)\boxtimes\sO_{\P^1}(4)\big|$. In particular, the projection to the first (or second) ruling of $Q$ shows that $\tilde C$ is hyperelliptic.

If $Q$ is the projective cone over a conic with the singular point
$O$, then $\tilde C$ must pass through $O$. Indeed, if $O\notin
\tilde C$, then $\tilde C$ is a Cartier divisor of $Q$. Since
$\Pic(Q)=\Z\cdot \sO_{Q}(1)$ and $\deg(\tilde C)=6$, we find that
$\tilde C$ is the (smooth) intersection of $Q$ with a cubic.
However, the smooth intersection of a cubic and a quadric in $\P^3$
should has genus 4 by the adjunction formula. This contradiction
shows $O\in \tilde C$. Then (after the blow-up of $Q$ at $O$), the
projection from $O$ to the conic provides a degree 2 morphism from
$\tilde C$ to the conic, showing that $\tilde C$ is hyperelliptic.

If $Q$ is the union of two projective plans, then $\tilde C$ is a
plan curve of degree 6, but then its genus should be 10 instead of
3, so this case can not happen at all.
\end{proof}

Since the smoothification $C$ can be chosen in a 3-dimensional
family $B$, and the normalization $\tilde C$ as well as the choice
of square root $L$ can also be carried over this base variety $B$
(by shrinking $B$ if necessary), we obtain a 3-dimensional family of
hyperelliptic curves mapping to the K3 surface $\Sigma$. Their
hyperelliptic involutions then yield a family of rational curves on
$\Sigma^{[2]}$ parameterized by $B$, where $\Sigma^{[2]}$ is the
Hilbert scheme of 0-dimensional subschemes of length 2 on $\Sigma$,
which is an irreducible holomorphic symplectic variety (\cf
\cite{MR730926}). Namely, we have the following diagram:
\begin{displaymath}
  \xymatrix{
  P\ar[d]\ar[r]^{f} & \Sigma^{[2]}\\
  B& \\
  }
\end{displaymath}
where $P$ is a $\P^1$-bundle over $B$ and $f$ is the natural
morphism. We now exclude this situation case by case:
\begin{itemize}
  \item If $f$ is generically finite, or equivalently dominant, then $\Sigma^{[2]}$ would be dominated by a ruled
  variety, thus uniruled, contradicts to the fact that its canonical bundle
  is trivial.
  \item If the image of $f$ is of dimension 3, \ie $f$ dominates a
  prime divisor $D$ of $\Sigma^{[2]}$. Let $\tilde D$ be a suitable
  resolution of singularities of $D$, such that the rationally connected quotient of $\tilde
  D$ is a morphism $q: \tilde D\to T$. After some blow-ups of $P$
  if necessary, we have the following commutative diagram:
  \begin{displaymath}
  \xymatrix{
  P\ar[d]\ar[r]^{f} & \tilde D\ar[d]^{q}\\
  B\ar[r]_{\bar f}& T \\
  }
\end{displaymath}
Firstly, we note that $\dim(T)\leq 2$ since $\tilde D$ is covered by
rational curves. Secondly, since the fibres of $q$ are rationally
connected, the morphism $q^*:H^{2,0}(T)\to H^{2,0}(\tilde D)$ is
surjective. However, let $\sigma\in H^{2, 0}(\Sigma^{[2]})$ be the holomorphic symplectic form (\cf \cite{MR730926}), then its restriction (more precisely, its pull-back) to $\tilde D$ is non-zero (since
$\dim(D)=3$). Therefore $H^{2,0}(T)\neq 0$, which implies in
particular $\dim(T)\geq 2$.\\
The above argument shows that $\dim(T)=2$. As a result, the fibres
of $q$ are unions of rational curves, and the dimension of the
fibres of $\bar f$ is at least 1. Therefore $f$ maps some
1-dimensional family of $\P^1$ into a union of (finitely many)
rational curves on $\Sigma^{[2]}$, which implies that all the
rational cuves in this 1-dimensional family are mapped into a same
rational curve, contradicting to the fact that all the rational curves
parameterized by $B$ are distinct from each other.

\item If the image of $f$ is contained in a surface of
$\Sigma^{[2]}$, then a dimension counting shows that through a
general point of the surface passes a 1-dimensional family of
rational curves, so it is actually a rational surface. As a result,
the corresponding points of $\Sigma^{[2]}$ are all equal up to
rational equivalence; in other words, the push-forwards of
the $g_2^1$'s of the family of hyperelliptic curves, viewed as
elements in $\CH_0(\Sigma)_{\Q}$, are all equal. However, for any
hyperelliptic curve $\tilde C$ in this family, let $\iota: \tilde
C\to \Sigma$ be the composition of the normalization $n$ with the natural inclusion, then
$$\iota_*\left(2g_2^1\right)=\iota_*\left(K_{\tilde
C}\right)=\iota_*\left(\iota^*\sO_{\Sigma}(C)|_{C}-n^*(Z)\right)=\sO_{\Sigma}(3)|_{C}-2Z,$$
which is non-constant by our construction, where $Z$ is the sum of
the nodes of $C$. This is a contradiction.
\end{itemize}
In conclusion, we have proved the claim $(\bigstar)$.

 Thanks to the isomorphism
$i^*: H^0(\P',\sO_{\P'}(2))\to H^0(\tilde C,2L)$, we can define the
morphism $\pi: \P'\to \P$, hence also the $q_0, \cdots, q_3$ as promised, by the following commutative diagram:
\begin{displaymath}
  \xymatrix{
  H^0(\P,\sO_{\P}(1)) \ar[r] \ar[d]^{\pi^*} & H^0(C, \sO_C(1))\ar[d]^{n^*}\\
  H^0(\P',\sO_{\P'}(2))\ar[r]^{i^*} & H^0(\tilde C,2L)}
\end{displaymath}
achieving the commutative diagram (\ref{embed}). Define
$S:=\pi^{-1}(\Sigma)$, then $\tilde C$ is a curve in $S$:
\begin{displaymath}
  \xymatrix{
  \tilde C\ar[r] \ar[d]^{n} & S \cart \ar[r] \ar[d]^{p} & \P'\ar[d]^{\pi}\\
  C \ar[r] & \Sigma \ar[r] & \P
  }
\end{displaymath}
By the adjunction formula, we have $$\sO_{S}(\tilde C)|_{\tilde C}+K_S|_{\tilde C}=K_{\tilde C}=n^*\bigg(\sO_\Sigma(C)|_C\otimes
\sO_C(-Z)\bigg).$$ 
Note that $\deg(S)=\deg(p)=\deg(\pi)=8$, hence $K_S=\sO_S(4)=p^*\sO_{\Sigma}(2)$. Therefore the above equality implies
$$\sO_{S}(\tilde C)|_{\tilde C}=n^*\left(\sO_C(1)\otimes \sO_C(-Z)\right).$$
Pushing this forward to $\Sigma$, we deduce that $$p_*({\tilde C}^2)=12c_{\Sigma}-2Z.$$ Since $Z\neq 16c_\Sigma$, ${\tilde C}^2$ is not
proportional to $h^2$ as desired, where $h=c_1(\sO_S(1))$.

In the above construction, we have not yet verified the smoothness
of $S$. However, as above the smoothification $C$ can be chosen in a 3-dimensional family $B$, and the construction of $\tilde C$, $L$,
$\pi$ and finally $S$ can also be carried over this base $B$ (by
shrinking $B$ if necessary). Generically, the family $\sS$ is the
transverse pull-back
\begin{displaymath}
 \xymatrix{
 \sS \ar[r] \ar[d]^{p_B} & \P'\times B \ar[d]^{\tilde\pi}\\
\Sigma\times B \ar[r] & \P\times B
  }
\end{displaymath}
which is generically smooth over $B$, thus a general fibre $S$ is
smooth.

\section{Higher powers of hypersurfaces of general type}
We now turn to the study of the smallest diagonal in higher self-products of a non-Fano hypersurface in the projective space. Our goal is to establish  Theorem \ref{thmB} in the introduction.

First of all, let us fix the basic set-up for this section. Let
$\P:=\P^{n+1}$ be the projective space, and $X$ be a general hypersurface in $\P$ of degree $d$, with $d\geq n+2$. Thus $X$ is an $n$-dimensional smooth variety. Since $K_X=\sO_X(d-n-2)$, when $d>n+2$, $X$ has ample canonical bundle thus is of general type; when $d=n+2$,
$X$ is of Calabi-Yau type and we will recover the results of
\cite{VoiFamilyK3}. Let $k:=d+1-n\geq 3$. Our first objective is to express in the Chow group of $X^k$ the class of the \emph{smallest} diagonal
 $$\delta_X:=\left\{(x,x,\cdots,x)~|~x\in
X\right\}$$ in terms of bigger diagonals. Since there will be various types of diagonals involved, we need
some systematic notation to treat them.
\begin{defi}\label{combdef}\upshape
   For any positive integers $s\leq r$,
  \begin{enumerate}
    \item we define the set $N^r_s$ to
  be the set of all possible \emph{partitions}\footnote{Here we
use the terminology `partition' in an unusual way.}  of $r$ elements
into $s$ non-empty \emph{non-ordered} parts. In other words:
  $$N^r_s:=\left\{\{1,2,\cdots,r\}\surj\{1,2,\cdots,s\}~ \text{surjective maps}\right\}/{\mathfrak{S}_s}.$$
    Here the action of the symmetric groups is induced from the
    action on the target. The action is clearly free, and we take the quotient in the naive (set-theoretical)
    way. We will view such a surjective map as some sort of
    \emph{degeneration} of $r$ points into $s$ points.
    In every equivalent class, we have a canonical representative
    $\dot\alpha:\{1,2,\cdots,r\}\surj\{1,2,\cdots,s\}$ such that $\dot\alpha(1),\dot\alpha(2),\cdots,
    \dot\alpha(r)$ is alphabetically minimal.
    \item For positive integers $t\leq s\leq r$, let $\alpha\in N^r_s$ and $\alpha'\in N^s_t$ be two partitions, then
   the \emph{composition} $\alpha\alpha'\in N^r_t$ is defined in
   the natural way, and obviously $\dot\alpha'\circ\dot\alpha$ is
   the minimal representative of $\alpha\alpha'$.
    \item For every $\alpha\in N^r_s$, and any set (or algebraic variety)
    $Y$, we have a natural morphism denoted still by $\alpha$,
    \begin{eqnarray*}
      \alpha: Y^s &\to& Y^r\\
      (y_1,y_2,\cdots, y_s)&\mapsto& (y_{\dot\alpha(1)}, y_{\dot\alpha(2)},\cdots, y_{\dot\alpha(r)})
    \end{eqnarray*}
   where $\dot\alpha$ is the minimal representative of $\alpha$
   defined above. Note that the morphism induced by their composition $\alpha\alpha'$ is
   exactly the composition of the morphisms induced by $\alpha$ and
   $\alpha'$, so there is no ambiguity in the notation $\alpha\alpha':Y^t\to
   Y^r$.
   \item The \emph{pull-back} of an $r$-tuple $\underline a=(a_1,
   a_2,\cdots, a_r)$ by an element $\alpha\in N^r_s$ is defined by
    $$\alpha^*(\underline a):= (b_1, \cdots, b_s),$$ where
    $b_j:=\sum_{\dot\alpha(i)=j}a_i$ for any $1\leq j\leq s$.
    We note that pull-backs are
    functorial: $(\alpha\alpha')^*(\underline a)=\alpha'^*\alpha^*(\underline
    a)$.
  \end{enumerate}
\end{defi}
Let us explain this definition in a concrete example: if $r=5, s=3$,
and the partition of $\{1,2,3,4,5\}$ is
$\alpha=\left(\{1,3\};\{4\};\{2,5\}\right)\in N^5_3$, then the
representative $\dot\alpha: 1\mapsto 1, 2\mapsto 2, 3\mapsto1,
4\mapsto3, 5\mapsto2$, and thus the corresponding morphism for any
$Y$ is
\begin{eqnarray*}
\alpha: Y^3& \to& Y^5\\
(y_1,y_2,y_3)&\mapsto&(y_1,y_2,y_1,y_3,y_2).
\end{eqnarray*}
And also the pull-back of a 5-tuple is $\alpha^*(a_1,\cdots,
a_5)=(a_1+a_3, a_2+a_5, a_4)$. If we have another $t=2$, $\alpha'\in
N^3_2$ defined by $(\{1,3\};\{2\})$, then $\alpha\alpha'$ is
$(\{1,2,3,5\};\{4\})$, and $\alpha\alpha':Y^2\to Y^5$ maps
$(y_1,y_2)$ to $(y_1,y_1,y_2,y_1,y_1)$.

This definition is nothing else but all the diagonal inclusions we
need in the sequel: for instance, the unique partition in $N^2_1$ is the
diagonal $\Delta_Y\subset Y\times Y$; the three elements in
$N^3_2$ is the so-called \emph{big} diagonals of $Y\times Y\times Y$
in the preceding section, and the unique partition in $N^r_1$ is the
smallest diagonal $\delta_Y\subset Y^r$. Note that the morphisms
$\alpha$ also induce morphisms from $Y^s\backslash\delta_Y$ to
$Y^r\backslash\delta_Y$.\\

Now return to our geometric setting. As in the preceding section,
for any integer $r\geq 2$, we define $$W_r:=
\left\{(y_1,y_2,\cdots,y_r)\in \P^{\times r}~|~y_i~ \text{are
collinear}\right\}\backslash\delta_{\P}.$$ In other words, if we
denote by $L\to G$ the universal line over the Grassmannian
$G:=\Gr(\P^1, \P)$, then in fact $W_r=
\underbrace{L\times_G\cdots\times_G L}_{r}\backslash \delta_{L}$. In
particular $W_r$ is a smooth variety with $\dim W_r=\dim G+r=2n+r$.

Similarly, for any $r\geq 2$, we define a closed subvariety $V_r$ of
$X^r\backslash\delta_X$ by taking the closure in
$X^r\backslash\delta_X$ of
$$V_r^o:=\left\{(x_1,x_2,\cdots,x_r)\in
X^r~|~x_i ~\text{are collinear and distinct}\right\}$$
$$V_r:=\bar{V_r^o}.$$
Then we can prove as in Lemma \ref{intersectionlemma} (in fact more
easily) or as in \cite{VoiFamilyK3} the following lemma. See also \cite{MR2339831}.
\begin{lemma} \label{intersectionlemma2} Consider the intersection of
$W_r$ and $X^r\backslash\delta_X$ in $\P^{\times
r}\backslash\delta_{\P}$. The intersection scheme has the following
irreducible component decomposition:
\begin{equation}\label{intersection2}
 W_r\cap (X^r\backslash\delta_X)= V_r\cup \bigcup_{2\leq
 s<r}\bigcup_{\alpha\in N^r_s}\alpha(V_s)
\end{equation}
Moreover, the intersection along each component is transversal, in
particular the intersection scheme is of pure dimension $2n$, that
is $\dim V_s=2n$ for any $s$. In particular, (\ref{intersection2})
also holds scheme-theoretically:
\begin{displaymath}
\xymatrix{ V_r\cup \bigcup_{2\leq
 s<r}\bigcup_{\alpha\in N^r_s}\alpha(V_s) \cart \ar[r] \ar[d] & X^r\backslash\delta_X \ar[d]\\
W_r \ar[r] &  \P^{\times r}\backslash\delta_{\P} }
\end{displaymath}
\end{lemma}

We now define some vector bundles on $V_r$ for any $r\in
\{2,3,\cdots, k\}$, which are analogues of the vector bundle $F$
defined in the preceding section. Let $S$ be the tautological rank 2
vector bundle on $W_r$, such that $p:\P(S)\to W_r$ is the
tautological $\P^1$-bundle, which admits $r$ tautological sections
$\sigma_i: W_r\to \P(S)$, where
$i=1,\cdots,r$. Let $q: \P(S)\to \P^{n+1}$ be the natural morphism.
We summarize the situation by the following diagram:
\begin{displaymath}
 \xymatrix{
\P(S) \ar[r]^{q} \ar[d]^{p} & \P\\
W_r\ar@/^/[u]<3ex>^{\sigma_i} \ar@/^/[u]<1ex>^{\cdots} }
\end{displaymath}
Let $B_i$ be the image of the section $\sigma_i$, which is a divisor
of $\P(S)$, for $i=1,\cdots,r$. For any $r$-tuple $\underline
a:=(a_1, a_2,\cdots, a_r)$ such that $\sum_{i=1}^ra_i=k$, we make the following constructions:
\begin{enumerate}
  \item A sheaf on $W_r$ by
\begin{equation*}
\tilde F\left(\underline a\right):=p_*\left(q^*\sO_{\P}(d)\otimes
\sO_{\P(S)}(-a_1B_1-\cdots-a_rB_r)\right).
\end{equation*}
As in Lemma \ref{FLemma}, we can prove that $\tilde F(\underline a)$
is a vector bundle on $W_r$ of rank $d+1-k=n$, with
 fibre $$\tilde F(\underline a)_{y_1\cdots y_r}=H^0(\P^1_{y_1\cdots y_r}, \sO(d)\otimes\sO(-a_1y_1-a_2y_2-\cdots-a_ry_r)).$$
 \item A rank $n$ vector bundle on $V_r$ by restriction: $$F(\underline a):=\tilde
F(\underline a)|_{V_r}.$$
\item An $n$-dimensional algebraic cycle on
$$\gamma_{\underline a}:=\gamma_{a_1,\cdots,
a_r}:=i_{r*}c_n\left(F(a_1, \cdots, a_r)\right)\in
\CH_n(X^r\backslash\delta_X),$$ where for any integer $r\geq 2$, we
denote the natural inclusion $i_r:V_r\to X^r\backslash\delta_X$.
\end{enumerate}
Recall that for any $2\leq s\leq r$ and any $\alpha\in N^r_s$, the morphisms $\alpha:X^s\backslash\delta_X\to X^r\backslash\delta_X$ induces a diagonal inclusion $\alpha: V_s\inj W_r$, whose formula is given by repeating some coordinates.
We observe the following relation:
\begin{lemma} \label{restriction}
For any $r$-tuple $\underline a:=(a_1, a_2,\cdots, a_r)$ such that
$\sum_{i=1}^ra_i=k$, the restriction of the vector bundle $\tilde F(\underline a)$ on $W_r$ to the image
  $\alpha(V_s)$ gives the vector bundle $F(\alpha^*(\underline a))$ on $V_s$, \ie
   \begin{equation}
\tilde F(\underline a)|_{\alpha(V_s)}=F(\alpha^*(\underline a)),
   \end{equation}
   where $\alpha^*(\underline a)$ is defined in Definition
   \ref{combdef}.
\end{lemma}
\begin{proof} To avoid heavy notation, let us explain in the simplest case
that $s=r-1$ and the partition $\alpha$ is given by
$(\{1,2\};\{3\};\{4\};\cdots;\{r\})$, then
  $\alpha: V_s\to W_r$ maps $(x_1, x_2,\cdots, x_s)$ to
  $(x_1,x_1,x_2,\cdots,x_s)$. Therefore the fibre of  $\tilde F(\underline
  a)|_{\alpha(V_s)}$ over the point $(x_1, x_2,\cdots, x_s)$ is
  exactly the fibre of $\tilde F(\underline a)$ over the point
  $(x_1,x_1,x_2,\cdots,x_s)$, which is
 $$H^0(\P^1_{x_1\cdots x_s},
  \sO(d)\otimes\sO(-(a_1+a_2)x_1-a_3x_2-\cdots-a_sx_r)).$$ It is
  nothing else but the fibre of $F(\alpha^*(\underline a))$ over
  the point $(x_1,x_2,\cdots,x_s)\in V_s$.
\end{proof}

Using the same trick as in Proposition \ref{basicequality}, we obtain the following recursive relations.

\begin{prop}\label{recursion}
The algebraic cycles $\gamma$ satisfy some recursive equalities: for
any integer $r\geq 3$ and any $r$-tuple $\underline a=(a_1, \cdots,
a_r)$ with $\sum_{i=1}^ra_i=k$, we have in
$\CH_n(X^r\backslash\delta_X)$,
\begin{equation}\label{recursionformula}
\gamma_{\underline a}+\sum_{2\leq s<r}\sum_{\alpha\in
N^r_s}\alpha_*\gamma_{\alpha^*(\underline a)}+ P_{\underline
a}(h_1,\cdots,h_r)=0,
\end{equation}
where $P_{\underline a}$ is a homogeneous polynomial of degree
$n(r-1)$ depending on $\underline a$. Moreover the starting data
are given by
\begin{equation}\label{starting}
\gamma_{a,b}=\prod_{i=0}^{n-1}\bigg((b+i)h_1+(n-1-i+a)h_2\bigg)
\end{equation}
 for any
$a+b=k$. Here $h_i$ is the pull-back of $h=c_1\left(\sO_X(1)\right)$
by the $i$-th projection.
\end{prop}
\begin{proof}
  Consider the cartesian square in Lemma \ref{intersectionlemma2}:
  \begin{displaymath}
\xymatrix{ V_r\cup \bigcup_{2\leq
 s<r}\bigcup_{\alpha\in N^r_s}\alpha(V_s) \cart \ar[r]^(.7){j_3} \ar[d]_{j_4} & X^r\backslash\delta_X \ar[d]_{j_2}\\
W_r \ar[r]_{j_1} &  \P^{\times r}\backslash\delta_{\P} }
\end{displaymath}
Thanks to Lemma \ref{intersectionlemma2}, there is no excess
intersection here, we thus obtain:
$$j_2^*j_{1*}c_n\left(\tilde F(\underline
a)\right)=j_{3*}j_4^*c_n\left(\tilde F(\underline a)\right).$$ In
the left hand side, $j_{1*}c_n\left(\tilde F(\underline a)\right)\in
\CH_{n+r}(\P^{\times r}\backslash\delta_{\P})=\CH_{n+r}(\P^{\times
r})$, which can be written as a homogeneous polynomial $-P(H_1,
\cdots, H_r)$ of degree $n(r-1)$. Applying $j_2^*$, $H_i$ restricts
to $h_i$. While in the right hand side,
\begin{eqnarray*}
j_4^*c_n\left(\tilde F(\underline a)\right)&=& \sum_{2\leq s\leq
r}\sum_{\alpha\in N^r_s} \alpha_*c_n\left(\tilde F(\underline
a)|_{\alpha(V_s)}\right)\\
&=&\sum_{2\leq s\leq r}\sum_{\alpha\in
N^r_s}\alpha_*\gamma_{\alpha^*(\underline a)} ~~~~\text{(by Lemma
\ref{restriction})}
\end{eqnarray*}
Putting them together, and noting that $N^r_r$ has only one element
inducing the identity map, we have (\ref{recursionformula}).\\
As
for the starting data, we only need to calculate $c_n(F(a,b))$. By
the same computations in the proof of Lemma \ref{M}, we find that on $V_2=X\times
X\backslash\Delta_X$, $$F(a,b)=\bigoplus_{i=0}^{n-1}\sO_X(i+b)\boxtimes
\sO_X(n-1-i+a),$$ and the formula (\ref{starting}) follows.
\end{proof}

Before we exploit the recursive formula (\ref{recursionformula})
further, we need the following easy lemma which allows us to
simplify the push-forwards by some diagonal maps. This lemma
essentially appeared in \cite{VoiFamilyK3} (Lemma 3.3), but for the
convenience of the readers we give a proof.
\begin{lemma}\label{voilemma}
Let $X$ be a hypersurface in $\P:=\P^{n+1}$ of degree $d$. Then in
$\CH_{n-1}(X\times X)$, $$d\Delta_*(h)=(i\times
i)^!(\Delta_{\P})=\sum_{j=1}^n h_1^jh_2^{n+1-j},$$ where $i:X\to \P$
is the natural inclusion, $\Delta:X\to X\times X$ is the diagonal
inclusion, and $h_i$ is the pull-back of $h:=c_1(\sO_X(1))$ by the
$i$-th projection.
\end{lemma}
\begin{proof}
  Consider the cartesian diagram:
  \begin{displaymath}
    \xymatrix{
    X\ar[r]^{\Delta} \cart \ar[d]_{i} & X\times X\ar[d]^{i\times i}\\
    \P \ar[r]_{\Delta_{\P}} & \P\times \P
    }
  \end{displaymath}
  Its excess normal bundle is exactly $N_{X/{\P}}\isom \sO_X(d)$.
  Therefore
  \begin{eqnarray*}
d\Delta_*(h)&=&\Delta_*\left(c_1\left(\sO_X(d)\right)\right)=\Delta_*\left(i^![\P]\cdot c_1(\text{excess normal bundle})\right)\\
&=& \Delta_*\left((i\times i)^![\P]\right)=(i\times
i)^!(\Delta_{\P}).
  \end{eqnarray*}
Finally, we know that in $\CH_{n+1}(\P\times \P)$, we have a
decomposition $\Delta_{\P}=\sum_{j=0}^{n+1}H_1^jH_2^{n+1-j}$.
Restricting this to $X\times X$, $H_i$ becomes $h_i$, and
$h_i^{n+1}=0$, proving the lemma.
\end{proof}

For the rest of this section, we will consider only $\Q$-coefficient
cycles. Let $c:=c_X:=\frac{1}{d}h^n\in \CH_0(X)_\Q$ be the 0-cycle
of degree 1, where $h=c_1\left(\sO_X(1)\right)$. Then note that
Lemma \ref{voilemma} implies the following simple equation in
$\CH_0(X\times X)_\Q$:
\begin{equation}\label{pushc}
  \Delta_*(c)=c\times c,
\end{equation}
where $\Delta:X\to X\times X$ is the diagonal inclusion.\\

To make use of the recursive formula (\ref{recursionformula}), we
need to introduce some terminology and notation. Let $r\geq 2$ be an
integer. For any non-empty subset $J$ of $\{1,2,\cdots, r\}$, define
the diagonal
$$\Delta_J:=\left\{(x_1, \cdots, x_r)\in X^r~|~x_j=x_{j'}, \forall j, j'\in J\right\},$$
which is a cycle of dimension $n(r+1-|J|)$. For example,
$\Delta_{\{1,2,\cdots, r\}}=\delta_X$ is the smallest diagonal, and
$\Delta_{\{i\}}=X^r$ for any $i$. Then for any proper subset $I$ of
$\{1,2,\cdots, r\}$, we can define the $n$-dimensional cycle
\begin{equation}\label{DefD1}
 D_I:=\Delta_{I^{c}}\cdot\prod_{i\in I}\pr_i^*c,
\end{equation}
where $I^c=\{1,2,\cdots, r\}\backslash
I$ is the complementary set. Informally, we could write
\begin{equation}\label{DefD2}
D_I=\left\{(x_1, \cdots, x_r)\in X^r~|~x_i=c, \forall i\in I;
x_j=x_{j'}, \forall j, j'\notin I\right\}.
\end{equation}
For example,
$D_{\emptyset}=\delta_X$;
$D_{\{1,2,\cdots,r-1\}}=\underbrace{c\times\cdots\times
c}_{r-1}\times X$. And for any $i$, $D_i:=D_{\{i\}}$ is called the
$i$-th \emph{secondary diagonal}. The crucial idea of the calculation of this
section is to focus on the coefficients of these secondary diagonals
in the expression of $\gamma$'s.

\begin{defi}\label{types}\upshape
An algebraic cycle in $\CH_n(X^r)_\Q$ is called of
\begin{itemize}
  \item \emph{Type $A$}: if it is a $\Q$-coefficient homogeneous polynomial
  in $h_1, \cdots, h_r$ of degree $n(r-1)$, such that each $h_i$ appears in every
  monomial, \ie if it is a $\Q$-linear combination of
     $\left\{\prod_{j=1}^rh_j^{m_j}; m_j>0,
     \sum_{j}m_j=n(r-1)\right\}_{i=1}^r$;\\

  And for any $0\leq j\leq r-1$:
  \item \emph{Type $B_j$}: if it is a $\Q$-linear combination of
  the cycles $D_I$ for $I$ proper subsets of $\{1,2,\cdots, r\}$ with $|I|=j$.
\end{itemize}
For example:
\begin{itemize}
  \item \emph{Type $B_0$}: if it is a multiple of the cycle
    $D_{\emptyset}=\delta_X$. We will not need this type.
  \item \emph{Type $B_1$}: if it is a $\Q$-linear combination of
    the secondary diagonals, \ie the cycles $D_{i}$ for $1\leq i\leq r$;
  \item $\cdots$
  \item \emph{Type $B_{r-1}$}: if it is a $\Q$-linear combination of
  $\left\{\prod_{j\neq i}h_j^n\right\}_{i=1}^r$;
\end{itemize}
We remark that these notions of types also make sense (except $B_0$
becomes zero) when viewed as cycles in
$\CH_n(X^r\backslash\delta_X)$ by restricting to this open subset.
We sometimes write Type $B_{\geq 2}$ for a sum of the form
$\text{Type}~B_2+\text{Type}~B_3+\cdots+\text{Type}~B_{r-1}$.
\end{defi}

Now we study the behaviors of the various types under push-forwards
by diagonal maps. In the simplest case, we have:
\begin{lemma}\label{push}
  Let $r\geq 2$ be an integer and $\alpha\in N^{r+1}_r$ be a
  partition, inducing a diagonal map $\alpha: X^r\backslash\delta_X\to
  X^{r+1}\backslash\delta_X$.  Then
  \begin{enumerate}
    \item $\alpha_*(\text{Type}~A)=\text{Type}~A$;
    \item
    $\alpha_*(\text{Type}~B_j)=\text{Type}~B_j+\text{Type}~B_{j+1}$,
     for any $1\leq j\leq r-1$. In particular,
    $$\alpha_*(\text{Type}~B_1)=\text{Type}~B_1+\text{Type}~B_2;$$
    $$\alpha_*(\text{Type}~B_2)=\text{Type}~B_{\geq2};$$
  \end{enumerate}
\end{lemma}
\begin{proof}
  For simplicity, one can suppose $\alpha:(x_1,x_2,\cdots, x_r)\mapsto (x_1,x_1,x_2,\cdots, x_r)$. Then the proofs are just some straightforward computations making use of Lemma
  \ref{voilemma} and (\ref{pushc}).
\end{proof}
Since any $\alpha\in N^r_s$ is a composition of several
one-step-degenerations treated in the above lemma, as a first
corollary of Proposition \ref{recursion}, any $(a_1,\cdots,a_r)$
with $\sum_{i=1}^ra_i=k$, $\gamma_{\underline a}$ is of the form:
$\text{Type}~A+\text{Type}~B_1+\cdots+\text{Type}~B_{r-1}.$

We need something more precise: our first objective is to determine
the coefficients of secondary diagonals, \ie cycles of type $B_1$,
in the expression of $\gamma_{\underline a}$ determined by the
recursive relations in Proposition \ref{recursion}. However, Lemma
\ref{push} tells us that the coefficients of $B_1$-cycles in the
$\gamma$'s for a certain $r$ are determined only by the coefficients
of $B_1$-cycles in $\gamma$'s for strictly smaller $r$'s. Now let us
work them out.

\begin{prop}\label{allgamma}
 Let $r\geq 2$ be an integer, $(a_1,\cdots, a_r)$ be an
$r$-tuple with $\sum_{i=1}^ra_i=k$, then
\begin{equation}\label{result}
\gamma_{a_1,\cdots,a_r}=\mu_r\sum_{i=1}^r\psi(a_i)D_i+
\text{Type}~B_{\geq 2}+\text{Type}~A,
\end{equation}
where the constants $\mu_r=(-1)^r(r-2)!$, and
$\psi(a_i):=d\cdot\frac{(n-1+\sum_{j\neq i}a_j)!}{(\sum_{j\neq
  i}a_j-1)!}=d\cdot\frac{(d-a_i)!}{(k-1-a_i)!}.$\\
In particular,
\begin{equation}\label{gamma111}
\gamma_{1^k}=(-1)^kd!\sum_{i=1}^kD_i+ \text{Type}~B_{\geq
2}+\text{Type}~A,
\end{equation}
where $1^k=(\underbrace{1,1,\cdots,1}_{k})$.
\end{prop}
\begin{proof}
Rewrite the recursive formula (\ref{recursionformula}): for $r\geq
3$,
\begin{equation*}
\gamma_{\underline a}+\sum_{2\leq s<r}\sum_{\alpha\in
N^r_s}\alpha_*\gamma_{\alpha^*(\underline a)}= \text{Type}~A+
\text{Type}~B_{r-1}=\text{Type}~A+ \text{Type}~B_{\geq 2};
\end{equation*}
and the starting data (\ref{starting}): for any $a+b=k$,
\begin{equation*}
\gamma_{a,b}=\frac{(b+n-1)!}{(b-1)!}h_1^n+\frac{(a+n-1)!}{(a-1)!}h_2^n+\text{Type}~A=\psi(a)D_1+\psi(b)D_2+\text{Type}~A.
\end{equation*}
Hence (\ref{result}) holds for $r=2$. Now we will prove the result
by induction on $r$. Thanks to Lemma \ref{push}, Type $A$ and Type
$B_{\geq 2}$ are preserved by $\alpha_*$, hence we can work throughout
this proof `modulo' these two types, and we will use `$\equiv$' in
the place of `$=$' to indicate such simplification.\\
In the recursive formula
\begin{equation}\label{recursive'}
\gamma_{\underline a}+\sum_{2\leq s<r}\sum_{\alpha\in
N^r_s}\alpha_*\gamma_{\alpha^*(\underline a)}\equiv0,
\end{equation}
to calculate a typical term $\alpha_*\gamma_{\alpha^*(\underline
a)}$, we first make an elementary remark that we can choose
\emph{any} representative of $\alpha$ to define the pull-back and
the induced push-forward morphism instead of sticking to the minimal
representative as we did before: $\alpha_*$, $\alpha^*$ compensate
each other for the effect of renumbering the parts. An example would
be helpful: let $\alpha\in N^6_4$ be as following:
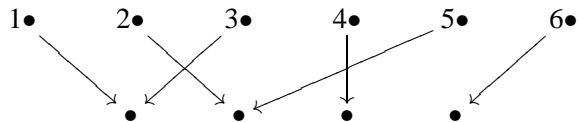
\begin{displaymath}
  \xymatrix{
  1 \bullet \ar[dr]& 2\bullet \ar[dr] &
  3\bullet\ar[dl] & 4\bullet \ar[d] &
  5\bullet \ar[dll] & 6\bullet\ar[dl]\\
  & \bullet &\bullet &  \bullet& \bullet &
  }
\end{displaymath}
then no matter how one renumbers the four points on the second row,
we always have (assuming the induction hypothesis (\ref{result}) for $r=4$):
\begin{eqnarray*}
\alpha_*\gamma_{\alpha^*(a_1, \cdots, a_6)} &=&
\mu_4\cdot\bigg(\psi(a_1+a_3)D_{\{1,3\}}+\psi(a_2+a_5)D_{\{2,5\}}+\psi(a_4)D_4+\psi(a_6)D_6\bigg)\\
&\equiv& \mu_4\cdot\bigg(\psi(a_4)D_4+\psi(a_6)D_6\bigg)
~~~~~~~~~~~~~~~ (\mod \text{Type}~A+ \text{Type}~B_{\geq2}).
\end{eqnarray*}
The second observation is that the contribution to Type $B_1$ of a
typical term $\alpha_*\gamma_{\alpha^*(\underline a)}$ with
$\alpha\in N^r_s$, is exactly the sum of $\mu_s\cdot\psi(a_i)D_i$
for those $i$ stays `isolated' in the partition defined by $\alpha$.
One can check this in the above example too, the isolated points are
4 and 6, and the contribution is
$\mu_4\cdot\bigg(\psi(a_4)D_4+\psi(a_6)D_6\bigg)$.\\
Therefore, the recursive formula (\ref{recursive'}) reads as (by the
induction hypothesis for all $2\leq s<r$):
\begin{equation}
\gamma_{\underline a}+\sum_{i=1}^r\sum_{2\leq
s<r}m_{r,s}\mu_s\cdot\psi(a_i)D_i\equiv0,
\end{equation}
where $m_{r,s}$ is the cardinality of the set $\{\alpha\in N^r_s~|~ i
~\text{stays isolated in the partition defined by}~ \alpha\}$, here
$i\in \{1,2,\cdots,r\}$, and obviously $m_{r,s}$ is independent of
$i$. However, by ignoring the isolated part, it is easy to see that
$m_{r,s}$ is exactly the cardinality of $N^{r-1}_{s-1}$. Therefore
to complete the proof, it suffices to show the following identity:
$$\mu_r+\sum_{2\leq s<r}\#N^{r-1}_{s-1}\cdot\mu_s=0,$$
which is an immediate consequence of the following elementary lemma.
\end{proof}

\begin{lemma} For any positive integer $m\geq 2$, we have
  $$\sum_{1\leq j\leq m}(-1)^j(j-1)!\cdot \#N^m_j=0.$$
\end{lemma}

\begin{proof} For any positive integer
$1\leq j\leq m$, let $S^m_j$ be the set of \emph{surjective} maps
from $\{1,2,\cdots, m\}$ to $\{1,2,\cdots, j\}$. By definition
\ref{combdef}, $N^m_j$ is the quotient of the action of
$\mathfrak{S}_j$ on $S^m_j$ induced by the action on the target.
This action is clearly free, thus $\#N^m_j=\frac{1}{j!}\#S^m_j$.
Denoting $s^m_j:=\#S^m_j$, we have to show the following identity:
\begin{equation}\label{elementary}
\sum_{1\leq j\leq m}(-1)^j\frac{1}{j}\cdot s^m_j=0.
\end{equation}
Now for any integer $1\leq l\leq m$, we consider the number of all
maps from $\{1,2,\cdots, m\}$ to $\{1,2,\cdots, l\}$, which is
obviously $l^m$. However on the other hand, we could count this
number by classifying these maps by the cardinality of their images:
the number of maps with  $\#$image$=j$ is exactly ${l \choose
j}\cdot s^m_j$. Hence,
$$l^m=\sum_{j=1}^m{l \choose j}\cdot s^m_j.$$
Since this identity holds for $l=0,1,\cdots, m$, it is in fact an
identity of polynomials of degree $m$:
$$T^m=\sum_{j=1}^m{T \choose j}\cdot s^m_j,$$ where $T$ is the variable.
Simplifying $T$ from both sides:
$$T^{m-1}=1+\sum_{j=2}^m s^m_j\cdot\frac{(T-1)(T-2)\cdots(T-j+1)}{j!}.$$
Let $T=0$, we obtain (\ref{elementary}), and the lemma follows.
\end{proof}

Now we relate the cycle $\gamma_{1^k}$ in (\ref{gamma111}) to a
geometric constructed cycle. Like before, let $F(X)$ be the variety
of lines of $X$. Since $X$ is general, $F(X)$ is a smooth variety of
dimension $2n-d-1=n-k$ if $k\leq n$, and empty if $k>n$. Define also
the subvariety of $X^k$
$$\Gamma:=\bigcup_{t\in
F(X)}\underbrace{\P^1_t\times\cdots\times\P^1_t}_{k},$$ which is of
dimension $n$ if $k\leq n$ and empty if $k>n$.

% But we remark that as
% long as $X$ is smooth, we can always define $\Gamma$ as an
% $n$-dimensional cycle of $X^k$, which is represented by its
% fundamental class if $X$ is general and $k\leq n$, and zero if
% $k>n$. The reason is that $F(X)$ and $\Gamma$, modulo rational
% equivalence, could have been defined in a purely intersection
% theoretical way, that is:
% $$F(X):=c_{d+1}(\Sym^d(S\dual));$$
% $$\Gamma:= i^!q_*p^*(F(X)),$$
% where the morphisms are defined in the following diagram, and $S$ is
% the universal rank 2 vector bundle on the Grassmannian $G=\Gr(\P^1,
% \P)$.
% \begin{displaymath}
%   \xymatrix{
%   & X^k\ar[d]^{i}\\
%  \P(S)^{\times_G k} \ar[r]_{q} \ar[d]_{p} & \P^{\times k}\\
%   G
%   }
% \end{displaymath}
% From now on, $\Gamma$ will be understood as such an $n$-dimensional
% cycle of $X^k$, satisfying $\Gamma=0$ when $k>n$.

Write
$\Gamma_o:=\Gamma|_{X^k\backslash\delta_X}\in
\CH_n(X^k\backslash\delta_X)$. As in Lemma \ref{gamma2},
\begin{lemma}\label{relategammas}
 Let $1^k$ be the $k$-tuple
$(1,1,\cdots,1)$, then $\Gamma_o=\gamma_{1^k}$ in
$\CH_n(X^k\backslash\delta_X)$.
\end{lemma}
\begin{proof}
The defining function of $X$ gives rise to a section of the rank $n$
vector bundle $F(1,1,\cdots,1)$ on $V_k$ by restricting to lines.
Its zero locus defines exactly the cycle $\Gamma_o$ in
$X^k\backslash \delta_X$. Then the geometrical meaning of top Chern
classes proves the desired equality.
\end{proof}

Combining the results of Proposition \ref{allgamma} and Lemma
\ref{relategammas}, we get a decomposition of the class of the smallest
diagonal of $X^k$, except that the multiple $\lambda_0$ appearing below could be zero.
\begin{prop}
  There exist rational numbers $\lambda_j$ for $j=0,\cdots, k-2$, and a symmetric
  homogeneous polynomial $P$ of degree $n(k-1)$, such that in
  $\CH_n(X^k)_{\Q}$ we have:
\begin{equation}\label{almostdecomp}
\Gamma=\sum_{j=0}^{k-2}\lambda_j\sum_{|I|=j}D_I+ P(h_1,\cdots,h_k),
\end{equation}
where $\lambda_1=(-1)^kd!$ is non-zero. More concretely,
\begin{equation}\label{almostdecomp'}
\Gamma=\lambda_0\delta_X+ (-1)^kd!\sum_{i=1}^kD_i+
\lambda_2\sum_{|I|=2}D_I+\cdots+\lambda_{k-2}\sum_{|I|=k-2}D_I+
P(h_1,\cdots,h_k).
\end{equation}
\end{prop}
\begin{proof}
Putting Lemma \ref{relategammas} into (\ref{gamma111}), we obtain
  $\Gamma_o=(-1)^kd!\sum_{i=1}^kD_i+ \text{Type}~B_{\geq
2}+\text{Type}~A$ in $\CH_n(X^k\backslash\delta_X)_{\Q}$. Thanks to
the localization exact sequence
$$\CH_n(X)_{\Q}\lra{\delta_*}\CH_n(X^k)_{\Q} \to\CH_n(X^k\backslash\delta_X)_{\Q}\to 0,$$ and the symmetry
of $\Gamma$, we can write it in the way as stated (remember that
Type $B_{k-1}$ is in fact a polynomial of the $h_i$'s).
\end{proof}

A second reflection on (\ref{almostdecomp}) or (\ref{almostdecomp'})
gives the main result of this section, which is a generalization of
Theorem \ref{V} in the introduction:
\begin{thm}\label{main3}
  Let $X$ be a general smooth hypersurface in $\P^{n+1}$ of degree $d$ with
  $d\geq n+2$. Let $k=d+1-n\geq 3$. Then one of the following two
  cases occurs:
  \begin{enumerate}
    \item There exist rational numbers $\lambda_j$ for $j=2,\cdots, k-1$, and a symmetric
  homogeneous polynomial $P$ of degree $n(k-1)$, such that in
  $\CH_n(X^k)_{\Q}$ we have:
\begin{equation}\label{case1}
\delta_X=(-1)^{k-1}\frac{1}{d!}\cdot\Gamma+\sum_{i=1}^kD_i+\sum_{j=2}^{k-2}\lambda_j\sum_{|I|=j}D_I+
P(h_1,\cdots,h_k),
\end{equation}
where $D_I$ is defined in (\ref{DefD1}) or (\ref{DefD2}).\\
Or

\item There exist a (smallest) integer $3\leq l<k$, rational numbers $\lambda_j$ for $j=2,\cdots, l-2$,
and a symmetric homogeneous polynomial $P$ of degree $n(l-1)$, such
that in
  $\CH_n(X^l)_{\Q}$ we have:
\begin{equation}\label{case2}
\delta_X=\sum_{i=1}^lD_i+\sum_{j=2}^{l-2}\lambda_j\sum_{|I|=j}D_I+
P(h_1,\cdots,h_l).
\end{equation}
  \end{enumerate}
Moreover, $\Gamma=0$ if $d\geq 2n$.
\end{thm}
\begin{proof}
In (\ref{almostdecomp}) or (\ref{almostdecomp'}), if
$\lambda_0=-\lambda_1(=(-1)^{k-1}d!)$ which is non-zero in
particular, then we can divide on both sides by $\lambda_1$ to get
(\ref{case1}) in Case 1,  up to a rescaling of the numbers $\lambda_i$ and the polynomial $P$.

 If $\lambda_0+\lambda_1\neq 0$, then we project both
sides onto the first $k-1$ factors. Since $\Gamma$ has relative
dimension 1 for this projection, it vanishes after the projection.
Therefore we get an equality in $\CH_n(X^{k-1})_{\Q}$ of the form:
\begin{equation*}
0=(\lambda_0+\lambda_1)\delta_X+ \lambda_1'\sum_{i=1}^kD_i+
\lambda_2'\sum_{|I|=2}D_I+\cdots+\lambda'_{k-2}\sum_{|I|=k-2}D_I+
P'(h_1,\cdots,h_{k-1}).
\end{equation*}
Dividing both sides by $\lambda_0+\lambda_1$, which is non-zero, we
get a decomposition of the smallest diagonal in $\CH_n(X^l)_{\Q}$ for
$l=k-1$:
\begin{equation}\label{reducing}
\delta_X=\lambda_1\sum_{i=1}^lD_i+\sum_{j=2}^{l-2}\lambda_j\sum_{|I|=j}D_I+
P(h_1,\cdots,h_l).
\end{equation}
by insisting the old notation $\lambda_i$ and $P$.

If such a decomposition does not exist for $l=k-2$, then by a
further projection to the first $k-2$ factors of (\ref{reducing}),
we find $\lambda_1=1$ in (\ref{reducing}) for $l=k-1$. Hence we
obtain a decomposition (\ref{case2}) in Case 2 for $l=k-1$.

If there does exist such decomposition for $l=k-2$, \ie we have an
identity of the form (\ref{reducing}) for $l=k-2$. By projecting to
the first $k-3$ factors, if such decomposition as (\ref{reducing})
for $l=k-3$ does not exist, then we find $\lambda_1=1$ in
(\ref{reducing}) for $l=k-2$. Hence we obtain a decomposition
(\ref{case2}) in Case 2 for $l=k-2$. If there does exist such
decomposition for $l=k-3$, we continue doing the same argument.

Since $H^{n,0}(X)\neq 0$, as in the proof of Lemma \ref{coefflemma},
the non-existence of a decomposition of the diagonal $\Delta_X\subset
X\times X$ implies that the minimal $l$ for the existence of a
decomposition of the form (\ref{reducing}) is at least 3. Therefore
the above argument must stop at some $l\geq 3$, and gives the
decomposition (\ref{case2}) in Case 2 for this minimal $l$.

As for the vanishing of $\Gamma$, we note that $d\geq 2n$ \iff
$k>n$, in which case we know that $\Gamma$ is empty.
\end{proof}

Now we draw the following consequence on the ring structure of
$\CH^*(X)_{\Q}$ from the above decomposition theorem, generalizing
Corollary \ref{VCor}:
\begin{thm}\label{main3cor}
Let $X$ be a general smooth hypersurface in $\P^{n+1}$ of degree $d$ with
  $d\geq n+2$. Let $m=d-n\geq 2$. Then for any strictly positive
  integers $i_1, i_2,\cdots, i_m\in \N^*$ with $\sum_{j=1}^mi_j=n$, the image
  $$\im\left(\CH^{i_1}(X)_\Q\times\CH^{i_2}(X)_\Q\times\cdots\times\CH^{i_m}(X)_\Q\lra{\bullet} \CH_0(X)_\Q\right)=\Q\cdot h^{n}$$
\end{thm}
\begin{proof}
  In our notation before, $m=k-1$. Let $z_j\in \CH^{i_j}(X)_{\Q}$ for $1\leq j\leq m$. By Theorem
  \ref{main3}, (\ref{case1}) or (\ref{case2}) holds. Suppose first
  that we are in Case 1, \ie (\ref{case1}). We view its both sides
  as correspondences from $X^m$ to $X$. Apply these
  correspondences to the algebraic cycle $z:=z_1\times\cdots\times z_m \in
  \CH^{n}(X^m)_{\Q}$:
  \begin{itemize}
    \item $\delta_{X*}(z)=z_1\bullet\cdots\bullet z_m$;
    \item $\Gamma_*(z)=0$ since $\Gamma_*(z)$ is represented by a
    linear combination of fundamental classes of certain
    subvarieties of dimension at least 1, but $\Gamma_*(z)$ is a
    zero-dimensional cycle, thus vanishes;
    \item $D_{I*}(z)= 0$ for any $I\neq \{k\}$;
    \item $D_{k*}(z)= \deg(z_1\bullet \cdots\bullet z_m)\cdot c_X$;
    \item $P(h_1, \cdots, h_{m+1})_*(z)$ is always proportional to
    $h^{n}$.
  \end{itemize}
  Therefore  $z_1\bullet \cdots\bullet z_m\in \Q\cdot
  h^{i_1+\cdots+i_m}$.\\
  If we are in Case 2, the same proof goes through.
\end{proof}

\begin{rmk}\label{rmkcor}\upshape
When $d=n+2$, this recovers the result of \cite{VoiFamilyK3} as in the first part of the paper. When $d>n+2$, the preceding theorem is actually predicted by the
Bloch-Beilinson conjecture, which roughly says that $\CH^i(X)_\Q$ is
controlled by the Hodge structures on the transcendantal parts of
$H^{2i}(X,\Q)$, $H^{2i-1}(X,\Q)$, $\cdots$, $H^i(X,\Q)$. However by
the Lefschetz hyperplane section theorem, the only non-Tate-type
Hodge structure of $H^*(X,\Q)$ is the middle cohomology $H^n(X,\Q)$.
Therefore, according to the Bloch-Beilinson conjecture, the smallest
$i$ such that $\CH^i(X)_{\Q}\supsetneqq \Q\cdot h^i$ is
$\lceil\frac{n}{2}\rceil$. Since in the corollary $\sum_{j=1}^mi_j=n$, the conjecture implies that there is at most one $j$ such that $z_j$ is not proportional to $h^{i_j}$. Now the above corollary follows from the easy fact that the intersection of any algebraic cycle $z$ with the hyperplane section class $h$ is always $\Q$-proportional to a power of $h$: write $\iota$ for the inclusion of the hypersurface $X$ into the projective space, then $z\bullet dh=\iota^*\iota_*(z)$, which is the pull-back of a cycle of the projective space, thus must be proportional to a power of $h$.
\end{rmk}

%\section{Questions:}
%\textbf{Questions:}
%\begin{enumerate}
%  \item Can we have a intrinsic proof (as in the case of K3 surfaces) of our
%  results? Here intrinsic means a proof independent of the
%  projective embedding.
%  \item Can we unify the two directions of generalizations of this
%  paper to deal with higher self-products of complete intersections of general type,
%  and thus obtain some consequences for the multiplicative structure of the Chow ring of a complete
%  intersection of general type?
%  \item How about the other Calabi-Yau varieties which are not complete
%  intersections?
%\end{enumerate}

%\footnoterule
\footnotesize
\labelsep .5em\relax
{%\setlength{\baselineskip}{0.7\baselineskip}
\bibliographystyle{plain}
\bibliography{biblio_fulie}

\begin{thebibliography}{10}

\bibitem{MR730926}
Arnaud Beauville.
\newblock Vari{\'e}t{\'e}s k{\"a}hleriennes dont la premi{\`e}re classe de
  chern est nulle.
\newblock {\em J. Differential Geom.}, 18(4):755--782 (1984), 1983.

\bibitem{MR2187148}
Arnaud Beauville.
\newblock On the splitting of the {B}loch-{B}eilinson filtration.
\newblock In {\em Algebraic cycles and motives. {V}ol. 2}, volume 344 of {\em
  London Math. Soc. Lecture Note Ser.}, pages 38--53. Cambridge Univ. Press,
  Cambridge, 2007.

\bibitem{MR2047674}
Arnaud Beauville and Claire Voisin.
\newblock On the {C}how ring of a {$K3$} surface.
\newblock {\em J. Algebraic Geom.}, 13(3):417--426, 2004.

\bibitem{MR923131}
Aleksandr~A. Be{\u\i}linson.
\newblock Height pairing between algebraic cycles.
\newblock In {\em {$K$}-theory, arithmetic and geometry ({M}oscow,
  1984--1986)}, volume 1289 of {\em Lecture Notes in Math.}, pages 1--25.
  Springer, Berlin, 1987.

\bibitem{MR714776}
Spencer Bloch and Vasudevan Srinivas.
\newblock Remarks on correspondences and algebraic cycles.
\newblock {\em Amer. J. Math.}, 105(5):1235--1253, 1983.

\bibitem{MR1675158}
Xi~Chen.
\newblock Rational curves on {$K3$} surfaces.
\newblock {\em J. Algebraic Geom.}, 8(2):245--278, 1999.

\bibitem{0913.14015}
Olivier Debarre and Laurent Manivel.
\newblock {Sur la vari{\'e}t{\'e} des espaces lin{\'e}aires contenus dans une
  intersection compl{\`e}te.}
\newblock {\em Math. Ann.}, 312(3):549--574, 1998.

\bibitem{MR1644323}
William Fulton.
\newblock {\em Intersection theory}, volume~2 of {\em Ergebnisse der Mathematik
  und ihrer Grenzgebiete. 3. Folge. A Series of Modern Surveys in Mathematics
  [Results in Mathematics and Related Areas. 3rd Series. A Series of Modern
  Surveys in Mathematics]}.
\newblock Springer-Verlag, Berlin, second edition, 1998.

\bibitem{MR2339831}
Fabien Herbaut.
\newblock Algebraic cycles on the {J}acobian of a curve with a linear system of
  given dimension.
\newblock {\em Compos. Math.}, 143(4):883--899, 2007.

\bibitem{MR0325625}
Jun ichi Igusa.
\newblock {\em Theta functions}.
\newblock Springer-Verlag, New York, 1972.
\newblock Die Grundlehren der mathematischen Wissenschaften, Band 194.

\bibitem{MR1265533}
Uwe Jannsen.
\newblock Motivic sheaves and filtrations on {C}how groups.
\newblock In {\em Motives ({S}eattle, {WA}, 1991)}, volume~55 of {\em Proc.
  Sympos. Pure Math.}, pages 245--302. Amer. Math. Soc., Providence, RI, 1994.

\bibitem{MR1669995}
Robert Laterveer.
\newblock Algebraic varieties with small {C}how groups.
\newblock {\em J. Math. Kyoto Univ.}, 38(4):673--694, 1998.

\bibitem{MR2960840}
Valeria~Ornella Marcucci and Gian~Pietro Pirola.
\newblock Points of order two on theta divisors.
\newblock {\em Atti Accad. Naz. Lincei Cl. Sci. Fis. Mat. Natur. Rend. Lincei
  (9) Mat. Appl.}, 23(3):319--323, 2012.

\bibitem{0557.14015}
Shigefumi Mori and Shigeru Mukai.
\newblock {The uniruledness of the moduli space of curves of genus 11.}
\newblock {}, 1983.

\bibitem{MR0204427}
D.~Mumford.
\newblock On the equations defining abelian varieties. {I}.
\newblock {\em Invent. Math.}, 1:287--354, 1966.

\bibitem{MR0249428}
D.~Mumford.
\newblock Rational equivalence of {$0$}-cycles on surfaces.
\newblock {\em J. Math. Kyoto Univ.}, 9:195--204, 1968.

\bibitem{MR1225267}
J.~P. Murre.
\newblock On a conjectural filtration on the {C}how groups of an algebraic
  variety. {I}. {T}he general conjectures and some examples.
\newblock {\em Indag. Math. (N.S.)}, 4(2):177--188, 1993.

\bibitem{MR1283872}
Kapil~H. Paranjape.
\newblock Cohomological and cycle-theoretic connectivity.
\newblock {\em Ann. of Math. (2)}, 139(3):641--660, 1994.

\bibitem{MR1997577}
Claire Voisin.
\newblock {\em Hodge theory and complex algebraic geometry. {II}}, volume~77 of
  {\em Cambridge Studies in Advanced Mathematics}.
\newblock Cambridge University Press, Cambridge, 2003.
\newblock Translated from the French by Leila Schneps.

\bibitem{MR2435839}
Claire Voisin.
\newblock On the {C}how ring of certain algebraic hyper-{K}{\"a}hler manifolds.
\newblock {\em Pure Appl. Math. Q.}, 4(3, part 2):613--649, 2008.

\bibitem{VoiFamilyK3}
Claire Voisin.
\newblock {Chow rings and decomposition theorems for $K3$ surfaces and
  Calabi-Yau hypersurfaces.}
\newblock {\em Geom. Topol.}, 16(1):433--473, 2012.

\end{thebibliography}
}

\noindent\sc{D\'epartement de Math\'ematiques et Applications, \'Ecole Normale Sup\'erieure, 45 Rue d'Ulm, 75230 Paris Cedex 05, France}\\
{\itshape E-mail address}: {\tt lie.fu@ens.fr}\\
{\itshape URL}: {\tt http://www.math.ens.fr/\texttildelow lfu/}

\end{document}